\documentclass[12pt]{amsart}
\usepackage[a4paper,margin=1.2in]{geometry}
\usepackage{graphics}
\usepackage[pdfencoding=auto,psdextra]{hyperref}
\usepackage{amsfonts, amsmath, amssymb, amsthm}
\usepackage{relsize}
\usepackage{epic}
\usepackage{graphicx,color}
\usepackage{changebar}
\usepackage{dsfont}

\usepackage{theoremref}
\usepackage{tikz-cd}

\newcommand{\ord}{{\rm ord} }

\newcommand{\Z}{\mathbb{Z}}

\newcommand{\Gal}{\text{Gal}}

\newcommand{\Ker}{\mathrm{ker}}

\newcommand{\coker}{\hbox{ Coker }}

\newcommand{\F}{\mathbb{F}}

\newcommand{\K}{\mathbb{K}}
\renewcommand{\L}{\mathbb{L}}
\renewcommand{\H}{\mathbb{H}}
\newcommand{\M}{\mathbb{M}}

\newcommand{\V}{\mathbb{V}}
\newcommand{\Q}{\mathbb{Q}}

\renewcommand{\C}{\mathbb{C}}
\newcommand{\N}{\mathbb{N}}

\newcommand{\p}{\mathfrak{p}}
\newcommand{\f}{\mathfrak{f}}
\newcommand{\g}{\mathfrak{g}}
\renewcommand{\a}{\mathfrak{a}}
\newcommand{\q}{\mathfrak{q}}
\newcommand{\I}{\mathcal{I}}

\renewcommand{\O}{\mathcal{O}}

\newtheorem{defi}{Definition}[section]
\newtheorem{theorem}[defi]{Theorem}
\newtheorem{proposition}[defi]{Proposition}
\newtheorem{corollary}[defi]{Corollary}
\newtheorem{lemma}[defi]{Lemma}
\newtheorem{remark}[defi]{Remark}
\newtheorem{definition}[defi]{Definition}

\def\eps{\varepsilon}

\author{Katharina M\"{u}ller}\address[K. M\"uller]{Mathematisches Institut
 der Universit\"at G\"ottingen \newline \begin{small}Bunsenstra{\ss}e 3-5\\37073 G\"ottingen, Germany\end{small}}
\email{katharina.mueller@mathematik.uni-goettingen.de}
\date{}

\title[Main conjecture for $p=2$]{The main conjecture for imaginary quadratic fields for the split prime $p=2$}
\begin{document}
\maketitle
\begin{abstract}
Let $\K$ be an imaginary quadratic field such that $2$ splits into two primes $\p$ and $\overline{\p}$. Let $\K_{\infty}$ be the unique $\Z_2$-extension of $\K$ unramified outside $\p$. Let $\L$ be an arbitrary finite abelian extension of $\K$. Let $\L_{\infty}=\K_{\infty}\L$ and let $\M$ be the maximal $p$-abelian, $\p$-ramified extension of $\L_{\infty}$. We set $X=\Gal(\M/\L_{\infty})$. In this paper we prove the Iwasawa main conjecture for the module $X$.
\end{abstract}
key words: $\p$-adic $L$-functions, elliptic units, Iwasawa Main conjecture.
\section{Introduction}
Let $\K$ be an imaginary quadratic field in which $2$ splits into two distinct primes $\p$ and $\overline{\p}$. By class field theory there exists exactly one $\Z_2$-extension $\K_{\infty}/\K$ which is unramified outside $\p$. Let $\L=\K(\mathfrak{fp^m})$ for some $m$ and 
$\L_{\infty}=\K_{\infty}\L$. We define $\L_n$ as the unique subextensions such that $[\L_n:\L]=2^n$. We will denote the Euler system of elliptic units in $\L_n$ by $C_n$. 
 
Let $\f$ be coprime to $\p$ and $\K\subset \L'\subset \L$ be an abelian extension such that $\L$ is the smallest ray class field of the type $\K(\f\p^m)$ containg $\L'$. Analogous to $\L_{\infty}$ we let $\L'_{\infty}=\K_{\infty}\L'$ and  $\L'_n$ be the intermediate fields. Let $U_n$ be the local units congruent to $1$ in $\L'_n$ and $U_{\infty}=\lim_{\infty\leftarrow n} U_n$. We define the elliptic units in $\L'_n$ by $C_n(\L')=N_{\L_n/\L'_n}(C_n)$. Let $E_n$ be the units of $\L'_n$ 
and define $\overline{E}=\lim_{\infty\leftarrow n}\overline{E}_n$. We define further $\overline{C}=\lim_{\infty\leftarrow n}\overline{C}_n$, where the overline denotes in both cases the $\p$-adic closure of the groups $E_n$ and $C_n$, respectively (i.e. we embed the groups $C$ and $E$ in the local units and consider their topological closure). 
We denote by  $A_n$ the $2$-part of the class group of $\L'_n$ and define $A_{\infty}=\lim_{\infty\leftarrow n} A_n$. Let further $\Omega$ be the maximal $2$-abelian $\p$-ramified extension of $\L'_{\infty}$. We will use the notation $X:=\Gal(\Omega/\L'_{\infty})$. 

There is a natural decomposition $Gal(\L'_{\infty}/\K)\cong H\times \Gamma'$, where $H=\Gal(\L'_{\infty}/\K_{\infty})$ and $\Gamma'\cong \Gal(\K_{\infty}/\K)$. We will fix once and for all such a decomposition. Let $\chi$ be a character of $H$ and $M$ an arbitrary $\Lambda=\Z_2[[\Gamma'\times H]]$-module. Let $\Z_2(\chi)$ be the extension of $\Z_2$ generated by the values of $\chi$ and define $M_{\chi}=M\otimes_{\Z_2[H]}\Z_2(\chi)$. So $M_{\chi}$ is the largest quotient on which $H$ acts via $\chi$. The modules $M_{\chi}$ are $\Lambda_{\chi}\cong \Z_2(\chi)[[T]]$-modules, where $T=\gamma-1$ for a topological generator $\gamma$ of $\Gamma'$. It is easy to verify that $X$, $A_{\infty},\overline{E}$ and $\overline{C}$ are $\Lambda$-modules. The main aim of this paper is to understand theit structure in more detail, i.e. to prove the following main conjecture.
\begin{theorem}
\label{mainmain}
$\textup{Char}(A_{\infty,\chi})=\textup{Char}((\overline{E}/\overline{C})_{\chi})$ and $\textup{Char}(X_{\chi})=\textup{Char}((U_{\infty}/\overline{C})_{\chi})$.
\end{theorem}
Let $\widetilde{\Gamma}=\Gal(\L_{\infty}/\L_{\infty}\cap \K(\f\p^2))\cong \Z_p$.
If we consider the field $({\L'_{\infty}})^{\widetilde{\Gamma}}$ we obtain an abelian extension of $\K$ contained in $\K(\f\p^2)$. As the projective limit does not depend on the finite level we start with we can without loss of generality assume that $\L'\subset \K(\f\p^2)$ for a suitable ideal $\f$ being coprime to $\p$. 
To prove the main result we will further use the following useful reduction step: Let $\f'$ be a principal ideal coprime to $\p$ in $\K$ such that $\omega_{\f'}=1$, where $\omega_{\f'}$ denotes the number of roots of unity of $\K$ congruent to $1$ mod $\f'$.
\begin{lemma}
\label{reduction}
If  Theorem \ref{mainmain} holds for $\K(\f'\p^{\infty}):=\cup _{n\in \N}\K(\f'\p^n)$, then it holds for every $\L'_{\infty}$.
\end{lemma}

Note that $\textup{Char}((U_{\infty}/\overline{C})_{\chi})$ can be seen as the Iwasawa-function $F(w,\chi)$ associated to the $\p$-adic $L$-function $L_{\p}(s,\chi)$ (compare with Corollary \ref{cor:lfunction}). So we could reformulate the second statement of Theorem \ref{mainmain} for $\L=\K(\f\p^2)$ as 
\begin{theorem}
\[\textup{Char}(X)_{\chi}=F(w,\chi^{-1}).\]
\end{theorem}
Theorem \ref{mainmain} was addressed before by Rubin in \cite{Rubin2} and \cite{Rubin} for $p\ge 3$ and $[\L':\K]$ coprime to $p$. Bley proved the conjecture in \cite{BL} for $p\ge 3$ and general ray class fields $\L'$ under the assumption that the class number of $\K$ is coprime to $p$. Oukhaba proved in \cite{OK} a divisor relation between $(\overline{E}/\overline{C})_{\chi}$ and $A_{\infty,\chi}$ under the assumption that the class number is coprime to $p$ and Vigui\'{e} proved in \cite{vg2} (preprint) such a relation including the case $p=2$ and even class number. Together with the fact that the $\mu$-invariant of $X$ vanishes this implies that the characteristic ideal of $A_{\infty,\chi}$ divides the characteristc ideal of $(\overline{E}/\overline{C})_{\chi}$. To underline the power of the Euler system of elliptic units we will reprove this relation using the fact that a certain $\mu$-invariant vanishes.

The most recent work on this problem is due to Kezuka \cite{Ke2} for $\K=\Q(\sqrt{-q})$ where $q$ is a prime congruent to $7$ modulo $8$. She proves the full main conjecture, including the definition of the pseudo-measure necessary for the definition of the $p$-adic $L$-function, in the case $\L'=\H$ the Hilbert class field of $\K$ and for all primes $p$ such that $p$ is split in $\K$, coprime to $q$ and to the class number of $\K$. Note that in Kezuka's case the definition of $\K$ ensures that $\K$ has odd class number - so here proof includes the prime $p=2$. In this article we drop the assumption that the class number has to be odd and allow $[\L:\K]$ to be even, i.e. we give a complete proof of the Main conjecture as stated in Theorem \ref{mainmain} for $p=2$ and any finite abelian extension $\L/\K$.

Our proof will follow closely the methods developed by Rubin and generalized by Bley and Kezuka. We will first construct a suitable measure on the group $\Gal(\L_{\infty}/\K)$ and use it to define a $\p$-adic $L$-function. This part of the paper is a summary of section 2 of \cite{mu}. Using properties of the Euler system of elliptic units developed by Rubin and Tchebotarev's Theorem we will reprove that $\textup{Char}(A_{\infty,\chi})$ divides $\textup{Char}(\overline{E}/\overline{C}_{\chi})$ (see also \cite{vg2} for similar results). In section 4 we will finish the proof by showing that they are generated by polynomials of the same degree and hence are equal - again using our $p$-adic $L$-functions. 

An analogue of the relation between the galois groups $\Gamma'$ and $\Gal(\L_{\infty}/\K)$ explained in section 3.2 holds for $p\ge 3$ as well. Thus, all results of section 3.2 can be proved for general $p$ and $\L$ as well. In fact most of them are in \cite{BL}. Thus, the proof given here can also be used to prove the main conjecture for general ray class fields $\L$ and any prime $p$ without the assumption that the class number of $\K$ has to be coprime to $p$. It is not stated here for the general case as it is given in \cite{BL} up to the slight modification in section 3.2 and to avoid technical case distinctions for example in section 3.1, where the statements for $p\ge 3$ and $p=2$ are actually different.

\section{$p$-adic measures}
Before we start defining the $\p$-adic measures we will need later we prove Lemma \ref{reduction}.
\begin{proof}[Proof of Lemma \ref{reduction}]
Let $M\in \{A_{\infty}, U_{\infty}/\overline{C}, \overline{E}/\overline{C},X\}$. We will write $M(\mathbb{J})$ when we want to make the field explicite. Let $\chi$ be a character of $\Gal(\L'_{\infty}/\K_{\infty})$. By inflation $\chi$ is also a character of $\Gal(\K(\f'\p^{\infty})/\K_{\infty})$. In particular, it is trivial on $\Gal(\K(\f'\p^{\infty})/\L'_{\infty})$. As $\f'$ is coprime to $\p$ and  none of the characteristic ideals is divisible by $2$ (this follows from Theorem \ref{mainmain} for $\K(\f'\p^{\infty})$ and the fact that $\textup{Char}(X)$ and $\textup{Char}(A_{\infty})$ are not divisible by $2$ as shown in Theorem \ref{mu1} and Corollary \ref{mu_2}) the $\Gal(\K(\f'\p^{\infty})/\L'_{\infty})$-invariant parts of $M(\K(\f'\p^{\infty}))$ are pseudoisomorphic to the norm $N_{\K(\f'\p^{\infty})/\L'_{\infty}}M(\K(\f'\p^{\infty}))$ which is pseudoisomorphic to $M(\L'_{\infty})$. Thus, we obtain $\textup{Char}(M(\L'_{\infty})_{\chi})=\textup{Char}(M(\L_{\infty})_{\chi})$.
\end{proof}
For the rest of the paper we will only consider the case $\L=\K(\f\p^2)$ for $\f$ being coprime to $\p$, principal and such that $\omega_{\f}=1$. Define $\F_n=\K(\f\p^{n})$ and note that $\L_n=\F_{n+2}$. We will use the notation $\F_0=\F=\K(\f)$. To define our elliptic units we will use the following exposition from \cite{mu}. 

\begin{lemma}\label{suitable curve}\cite[Lemma 2]{mu} There exists an elliptic curve $E/\F$ which satisfies the following properties.
\begin{itemize}
\item[a)] $E$ has CM by the ring of integers $\mathcal{O}_{\K}$ of $\K$;
\item[b)] $\F\left(E_{tors}\right)$ is an abelian extension of $\K$;
\item[c)] $E$ has good reduction at primes in $\F$ lying above $\p$.
\end{itemize}
\end{lemma}

Let $\phi$ be a Gr\"ossencharacter of $\K$ of infinity type $(1,0)$ and conductor $\f$. Let $E/\F$ be the elliptic curve defined in Lemma \ref{suitable curve}. Then we can assume that the Gr\"ossencharacter $\psi$ associated to $E$ satisfies
\[ \psi_{E/\F}=\phi \circ N_{\F/\K}.\] 

In the sequel we have to describe the points on our elliptic curve explicitely. Therefore, we fix once and for all a minimal Weierstrass model of $E$.\begin{eqnarray}\label{wmodel}
y^2+a_1xy +a_3y=x^3+a_2 x^2 +a_4 x+a_6.
\end{eqnarray}
Using the fact that $E$ has good reduction at all primes above $\p$ we can assume that the discriminant $\Delta(E)$ is coprime to $\p$. Choosing a suitable embedding $\F\hookrightarrow \C$ we can further assume that there is a complex number $\Omega_{\infty}$ such that the period lattice $\mathcal{L}$ of $E$ satisfies $\mathcal{L}=\Omega_{\infty}\mathcal{O}(\K)$ (see \cite{CoGo83} for details).

\smallskip Let $\sigma\in \Gal(\F/\K)$ be an arbitrary element. Then $\sigma$ acts on the coefficients of the model \eqref{wmodel} and defines another curve $E^{\sigma}$ over $\F$. The curve $E^{\sigma}$ satisfies point a)-c) of Lemma \ref{suitable curve}. From point b) we obtain that the two curves $E$ and $E^{\sigma}$ have the same Gr\"ossencharacter. In particular, all the $E^{\sigma}$ are isogenous to each other.

\smallskip Let $\mathfrak{a}$ be an ideal in $\K$ coprime to $\f\p$ and $a\in\mathfrak{a}$ an arbitrary element. By point a) of Lemma \ref{suitable curve} we see that the multiplication by $a$ is a well defined endomorphism of $E$ and we can consider its kernel $E_a$.
We define further $E_{\mathfrak{a}}=\cap_{a\in\mathfrak{a}}E_a$. Let $\sigma_{\mathfrak{a}}$ be the Artin symbol of $\mathfrak{a}$ in $\Gal(\F/\K)$. Then the main theorem of complex conjugation allows us to define an isogeny $\eta_{\sigma}(\a) : E^{\sigma} \to E^{\sigma\sigma_{\a}}$ over $\F$, of degree $N(\a)$. This isogeny has the property that for every $\mathfrak{g}$ coprime to $\a$ and any $u\in E_{\g}$ we have
\[ \sigma_{\a}(u)=\eta_{\sigma}(\a)(u).\] Moreover, the kernel of this isogeny is precisely the subgroup $E^{\sigma}_{\a}$ (see \cite[proof of Lemma 4]{CoGo83} ). Whenever $\sigma$ is trivial we will drop the subscript $\sigma$ and write $\eta(\mathfrak{a})$ instead of $\eta_{id}(\mathfrak{a})$.

\smallskip Let $\omega$ be the Neron-differential associated to the model \eqref{wmodel} of $E$ then
\[ \omega= \frac{dx}{2y+a_1x+a_3}.\]
Due to \cite[p. 341]{CoGo83}, there exists a unique $\Lambda(\mathfrak{a}) \in \F^{\ast}$ such that
\begin{eqnarray}
\omega^{\sigma_{\mathfrak{a}}} \circ \eta(\mathfrak{a})=\Lambda(\mathfrak{a})\omega.
\end{eqnarray}
This defines a map $\Lambda\colon\{\textup{ideals coprime to }\f\}\to \F^{\times}$ satisfying the cocyle condition. Therefore it can be extended to a cocyle of all fractional idelas coprime to $\f$. This cocyle $\Lambda$ plays also an important role in determining the period lattice for the curve $E^{\sigma_{\a}}$. Let $\mathcal{L}_{\sigma_\mathfrak{a}}$ be the period lattice of $E^{\sigma_{\mathfrak{a}}}$. Then 
\begin{eqnarray}
\Lambda(\mathfrak{a})\Omega_{\infty}\mathfrak{a}^{-1}=\mathcal{L}_{\sigma_{\a}}.
\end{eqnarray}
(see \cite[p. 342]{CoGo83} for details).

\smallskip Recall that we choose an embedding $\overline{\Q}$ into $\C$ such that $\mathcal{L}=\Omega_{\infty}\mathcal{O}(\K)$. Choose a place $v$ above $\p$ induced by this embedding and let $\mathcal{I}_{\p}$ be the ring of intergers of the maximal unramified extension of $\F_v$. For ervery $\sigma\in \Gal(\F/\K)$ we denote by $\widehat{E^{\sigma,v}}$ the formal group given by the kernel of the reduction modulo $v$ of $E^{\sigma}/\F$. Note that the formal parameter of this group is $t_{\sigma}=-x_{\sigma}/y_{\sigma}$. If $\sigma$ is trivial we ommit the sub- and superskript $\sigma$.  With these notations we obtain:

\begin{lemma}\label{formal iso} \cite[Lemma 3]{mu} There exists an isomorphism $\beta^v$ between the formal multiplicative group $\widehat{\mathbf{G}}_m$ and the formal group $\widehat{E^{v}}$, which can be written as a power series $t=\beta^v(w) \in \mathcal{I}_{\p}[[w]]$ .
\end{lemma}

We fix once and for all such an isomorphism and denote the coefficient of $w$ in this power series by $\Omega_v$. The isogenies $\eta(\mathfrak{a})$ induce homomorphisms $\widehat{\eta(\mathfrak{a})}\colon \widehat{E^v}\to \widehat{E^{\sigma_{\mathfrak{a}},v}}$. As long as $\mathfrak{a}$ is coprime to $\f$ they are even isomorphisms. Define $\beta^v_{\a}=\widehat{\eta(\a)}\circ \beta^v$. If we denote the first coefficient of $\beta^v_{\a}$ by $\Omega_{\a,v
}$ we obtain from \cite[Lemma 6]{CoGo83} that $\Omega_{\a,v}=\Lambda(\a)\Omega_v$.

\smallskip These notations and definitions allow us to define our basic rational function: Let $\alpha$ be an integral element in $\K$ that is coprime to $6$ and not a unit. Let $P^{\sigma}$ be a generic point on $E^{\sigma}$ and denote its $x$- and $y$-coordinates by $x(P^{\sigma})$ and $y(P^{\sigma})$. Then we define
\[\xi_{\alpha,\sigma}(P^{\sigma})=c_{\sigma}(\alpha)\prod\limits_{S \in V_{\alpha,\sigma}} \left(x(P^{\sigma})-x(S)\right),\]
where $V_{\alpha,\sigma}$ is a set of representatives of the non-zero $\alpha$-division points on $E^{\sigma}$ modulo $\{\pm 1\}$ and $c_{\sigma}(\alpha)$ is a canonical 12th root in $\F$ of $\Delta(\alpha^{-1} \mathcal{L}_{\sigma})/\Delta(\mathcal{L}_{\sigma})^{N_{\K/\Q}(\alpha)}$ (here $\Delta$ stands for the Ramanujan's $\Delta$-function) .-see also \cite[Appendix, Proposition 1]{Co91} and \cite[Appendix, Theorem 8]{Co91}.
Recall that we assumed that $\f=(f)$ is principal. By the definition of $\mathcal{L}$ we see that $\rho=\Omega_{\infty}/f$ defines a primitive $\f$ division point on $E$. 

Using this we define the rational function
\[\xi_{\alpha,\sigma,Q}(P^{\sigma})=\xi_{\alpha,\sigma}(P^{\sigma}+\Lambda(\a)\rho),\]
where $\a$ is an integral ideal coprime to $\f$ such that $\left(\frac{\a}{\F/\K}\right)=\sigma$. We fix a set of ideals $\mathfrak{C}_0$ such that the Artin symbols of $\a\in\mathfrak{C}_0$ run through every element of $\Gal(\F/\K)$ exactly once and define \[ Y_{\alpha,\a}(P^{\sigma})=\frac{\xi_{\alpha,\sigma,Q}(P^{\sigma})^p}{\xi_{\alpha,\sigma\sigma_{\p},Q}(\eta_{\sigma}(\p)(P^{\sigma}))}.\]
We obtain the following result (\cite[Lemma 4]{mu} and \cite[equation 15]{mu}): \begin{lemma}\label{ymeasure} For an integral ideal $\a$ of $\O_{\K}$ coprime to $\f$, let $\sigma_{\a}$ denote the Artin symbol of $\a$ in $\Gal(\F/\K)$. Then the series $Y_{\alpha,\a}(t_{\sigma_{\a}})$ lies in $1+\mathfrak{m}_v[[t_{\sigma_{\a}}]]$ and the series $h_{\alpha,\a}(t_{\sigma_{a}}):=\frac{1}{2}\log(Y_{\alpha,\a}(t_{\sigma_{\a}}))$ has coefficients in $\O(\F_v)$. Further, the function $Y_{\alpha,\a}(P^{\sigma})$ satisfies the relation $\prod_{R\in E^{\sigma}_{\p}}Y_{\alpha,\a}(P^{\sigma}\oplus R)=0$.
\end{lemma}

\smallskip
There is a one to one correspondence between $\mathcal{I}_{\p}$-valued measures on $\Z_p^{\times}$ and the ring of power series $\mathcal{I}_{\p}[[T]]$ given by Mahler's Isomorphism
\[F_{\nu}(w)=\int_{\Z_2}(1+w)^xd\nu(x).\] 
For every $\a\in\mathfrak{C}_0$ we can consider the power series $\mathcal{B}_{\alpha,\a}(w)=h_{\alpha,\a}\left(\beta^v_{\a}(w)\right)$. Using Lemma \ref{ymeasure} we obtain that these series correspond to measures on $\Z_2$. We will denote these measures by $\nu_{\alpha,\a}$. Using \cite[Lemma 1.1]{sinnott} together with the second claim of Lemma \ref{ymeasure} we see that the measures $\nu_{\alpha,\a}$ coincide with their restriction to $\Z_2^{\times}$. 
As $\Gal(\L_{\infty}/\F)\cong \Z_2^{\times}$ we can see the measures $\nu_{\alpha,\a}$ as measures on $\Gal(\L_{\infty}/\F)$. If we extend them by zero outside $\Gal(\L_{\infty}/\F)$ we can even see them as measures on $\Gal(\L_{\infty}/\K)$. For $\a\in \mathfrak{C}_0$ let $\nu_{\alpha,\a}\circ\sigma_{\a}$ be the pushforward measure of $\nu_{\alpha,\a}$ on $\sigma_{\a}^{-1}\Gal(\L_{\infty}/\F)$. Then we can define 
\[ \nu_{\alpha}:=\sum\limits_{\a\in \mathfrak{C}_0} \nu_{\alpha,\a} \circ \sigma_{\a},\]
which is by the choice of $\mathfrak{C}_0$ a measure on $\Gal(\L_{\infty}/\K)$.
 The measures $\nu_{\alpha}$ have nice interpolation propertiees with respect to $L$-functions when it comes to integration of characters of the form $\varepsilon=\chi\phi^k$ for a character of finite order $\chi$. Let $\chi$ be a character of conductor $\g\p^n$ with $\g\mid \f$ and consider the set 
 \[S=\left\lbrace \gamma \in \Gal\left(\K(\f\p^n\overline{\p}^{\infty})/\K\right): \left. \gamma\right|_{\K(\f\overline{p}^{\infty})}=\left(\frac{\K(\f\overline{\p}^{\infty})/\K}{\p^n}\right)\right\rbrace.\]
 With this definition we can define the Gauss-like sum \[ G(\eps)=\frac{\phi^k(\p^n)}{p^n}\sum\limits_{\gamma \in S} \chi(\gamma) \zeta_{p^n}^{-\gamma}.\]
 Thiese notations allow us to state:
 \begin{theorem}\cite[Theorem 4]{mu}\label{shalit412} Let $\mathcal{D}_{\p}=\mathcal{I}_{\p}(\zeta_m)$, where $\zeta_m$ denotes an $m$-th root of unity and $m=\vert H\vert$.
 Then there exists a unique measure $\nu$ on $\Gal(\F_{\infty}/\K)$ taking values in $\mathcal{D}_{\p}$ such that for any $\eps=\phi^k \chi$, with $k\geq 1$ and $\chi$ a character of conductor dividing $\f\p^n$ for some $n \geq 0$, one has
\[ \Omega_v^{-k} \hspace{-0.5cm} \int\limits_{\Gal(\L_{\infty}/\K)} \hspace{-0.5cm} \eps \, d\nu =\Omega_{\infty}^{-k}(-1)^k (k-1)!f^k u_{\chi} G(\eps)\left(1-\frac{\eps(\p)}{p}\right) L_{\f}(\overline{\eps},k),\]
with a unit $u_{\chi}$ only depending on $\chi$. Further, $\nu_{\alpha}=(N\alpha-\sigma_{\alpha})\nu$. 

If $\chi$ is a non-trivial
character of $H$ it can be extended linearly to the ring of $\mathcal{D}_{\p}$ valued measures on $\Gal(\L_{\infty}/\K)$. It follows that the $\chi(\nu_{\alpha}/\nu)$ generate the trivial ideal.
\end{theorem}
\begin{proof}
Only the first claim is stated in Theorem 4 of \cite{mu}. But the other claims are stated as intermediate steps in the proof.
\end{proof}

 To prove the main conjecture we will not only need measures on $\Gal(\L_{\infty}/\K)$ but also on $\Gal(\K(\g\p^{\infty})/\K)$ for $\g\mid \f$. Therefore we define the pseudomesure 
\begin{eqnarray}\label{smallermeas}
\nu(\g) := \left. \nu(\f)\right|_{\Gal(\K(\g\p^{\infty})/\K)} \prod\limits_{\stackrel{\mathfrak{l} \mid \f}{\mathfrak{l} \nmid \g}} \left(1-\left(\left. \sigma_{\mathfrak{l}}\right|_{\K(\g\p^{\infty})}\right)^{-1}\right)^{-1},
\end{eqnarray}
where $\left. \nu(\f)\right|_{\Gal(\K(\g\p^{\infty})/\K)}$ is the measure on $\Gal(\K(\g\p^{\infty})/\K)$ induced from $\nu(\f)$.
Note that these pseudomesures are in fact measures as soon as $\g\neq (1)$, while $(1-\sigma)(\nu(1))$ is actually a measure for every $\sigma$ in $\Gal(\K(\p^{\infty})/\K)$. It is an easy verification that in the case $\omega_{\g}=1$ the measure $\nu(\g)$ is actually the measure one would obtain by starting with an elliptic curve $E'/\K(\g)$ and do all the constructions we did so far directly for $\Gal(\K(\g\p^{\infty})/\K)$ (compare with \cite[comments after II 4.12]{Shalit}).

Having all these definitions in place allows us to define our $\p$-adic $L$-function.
\begin{definition}\label{padicdef} Fix an isomorphism
\[ \kappa:\Gamma' \to 1+4\Z_2,\]
and let $\chi$ be a character of $H$. We denote by $\g_{\chi}$ the prime to $\p$-part of its conductor and define the $\p$-adic $L$-function of the character $\chi$ as
\[L_{\p}(s,\chi)= \int\limits_{\Gal(\K(\g_{\chi}\p^{\infty})/\K)} \chi^{-1} \kappa^{s} d\nu(\g_{\chi}) \quad\text{if } \chi \neq 1;\]
\[L_{\p}(s,\chi)=\int\limits_{\Gal(\K(\p^{\infty})/\K)} \chi^{-1} \kappa^{s} \, \, d\left((1-\gamma)\nu(1)\right) \quad\text{if } \chi= 1.\]
\end{definition}

\section{Elliptic units and Euler Systems}
It is well known that for every $\mathfrak{m}$ torsion point $P^{\sigma}_{\mathfrak{m}}$ on $E^{\sigma}$ the elements $\xi_{\alpha,\sigma}(P^{\sigma}_{\mathfrak{m}})$ are contained in $\K(\mathfrak{m})$ \cite[Proposition II 2.4]{Shalit}. The following proposition will be very useful in the course of our proof. 
\begin{proposition}
\label{prop:norms}
Let $\mathfrak{m}$ be an ideal coprime to $\alpha\f$ and $P\in E^{\sigma}_{\mathfrak{m}}$ a primitive $\mathfrak{m}$-division point. Let $\mathfrak{r}$ be a prime and $\mathfrak{m}=\mathfrak{r}\mathfrak{m}'$ with $\K(\mathfrak{m}')\neq \K(1)$. Then
\[
N_{\K(\mathfrak{m})/\K(\mathfrak{m}')}\xi_{\alpha,\sigma}(P)=
\begin{cases}
\xi_{\alpha,\sigma\sigma_{\mathfrak{r}}}(\eta_{\sigma}(\mathfrak{r})P)\quad & \mathfrak{r}\mid\mathfrak{m}'\\
\xi_{\alpha,\sigma\sigma_{r}}(\eta_{\sigma}(\mathfrak{r})P)^{1-\textup{Frob}_{\mathfrak{r}}^{-1}} \quad  & \mathfrak{r}\nmid \mathfrak{m}'
\end{cases}
\]
\end{proposition}
\begin{proof}
This proof follows \cite[Proposition 4.3.2]{Kezuka-thesis}.
The unit group $\mathcal{O}^{\times}=\mathcal{O}(\K)^{\times}$ has exactly two elements. Hence, the map $\mathcal{O}^{\times}\to (\mathcal{O}/\mathfrak{m})^*$ is injective. It follows that the kernel of the projection
\[\phi\colon (\mathcal{O}/\mathfrak{m})^{\times}/\mathcal{O}^{\times}\to (\mathcal{O}/\mathfrak{m}')^{\times}/\mathcal{O}^{\times}\] is isomorphic to the kernel of 
\[\phi'\colon (\mathcal{O}/\mathfrak{m})^{\times}\to (\mathcal{O}/\mathfrak{m}')^{\times}.\]
Hence,
\[[\K(\mathfrak{m}):\K(\mathfrak{m}')]=\begin{cases}
N\mathfrak{r}-1\quad &\mathfrak{r}\nmid \mathfrak{m}'\\
N\mathfrak{r}\quad& \mathfrak{r}\mid\mathfrak{m}'
\end{cases}.\]
The conjugates of $P$ under $\Gal(\K(\mathfrak{m})/\K(\mathfrak{m}'))$ are the set \[\{P+Q\mid Q\in E^{\sigma}_{\mathfrak{r}}\textup{ such that }P+Q\notin E^{\sigma}_{\mathfrak{m}'}\}\] if $\mathfrak{r}\nmid \mathfrak{m}'$ and \[\{P+Q\mid Q\in E^{\sigma}_{\mathfrak{r}}\}\] if $\mathfrak{r}\mid \mathfrak{m}'$.
In the first case there is exactly one $\mathfrak{r}$-torsion point $Q_0$ such that $P+Q_0$ is contained in $E^{\sigma}_{\mathfrak{m}'}$. We obtain 
\[\xi_{\alpha,\sigma}(P+Q_0)N_{K(\mathfrak{m})/\K(\mathfrak{m}')}\xi_{\alpha,\sigma}(P)=\prod_{Q\in E^{\sigma}_{\mathfrak{r}}}\xi_{\alpha,\sigma}(P+Q)=\xi_{\alpha,\sigma\sigma_{\mathfrak{r}}}(\eta_{\sigma}(\mathfrak{r})P).\]
By the definition of $\eta$ we obtain further that \[\xi_{\alpha,\sigma}(P+Q_0)^{\textup{Frob}_{\mathfrak{r}}}=\xi_{\alpha,\sigma\sigma_{\mathfrak{r}}}(\eta_{\sigma}(\mathfrak{r})(P+Q_0))=\xi_{\alpha,\sigma\sigma_{\mathfrak{r}}}(\eta_{\sigma}(\mathfrak{r})P),\] which implies the claim in this case.

\smallskip In the case $\mathfrak{r}\mid\mathfrak{m}$ we obtain the claim directly from  
\[
N_{K(\mathfrak{m})/\K(\mathfrak{m}')}\xi_{\alpha,\sigma}(P)=\prod_{Q\in E^{\sigma}_{\mathfrak{r}}}\xi_{\alpha,\sigma}(P+Q)=\xi_{\alpha,\sigma\sigma_{\mathfrak{r}}}(\eta_{\sigma}(\mathfrak{r})P).\]
\end{proof}

Before we can define our Euler system we still need one further concept. Let $S_{n,l}$ be the set of square free ideals of $\mathcal{O}$ that are only divisible by prime ideals $\mathfrak{q}$ satisfying the following two conditions
\begin{itemize}
\item[i)] $\mathfrak{q}$ is totally split in $\L_{n}=\K(\f\p^{n+2})$
\item[ii)]$N\mathfrak{q}\equiv 1\mod 2^{l+1}$
\end{itemize}
\begin{lemma}
\label{field}
Let $\H_n=\K(\p^{n+2})$.
Given a prime $\mathfrak{q}$ in $S_{n,l}$ there exists a cyclic extension $\H_n(\mathfrak{q})/\H_n$ of degree $2^l$ inside $\H_n\K(\mathfrak{q})$. Furthermore, $\H_n(\mathfrak{q})/\H_n$ is totally ramified at the primes above $\mathfrak{q}$ and unramified outside $\mathfrak{q}$. Let $\mathbb{V}$ be any subfield $\mathbb{H}_n\subset \mathbb{V}\subset \L_n$ and $\mathbb{V}(\mathfrak{q})=\H_n(\q)\mathbb{V}$ then $\Gal(\mathbb{V}(\q)/\mathbb{V})\cong \Gal(\H_n(\q)/\H_n)$ and the ramification behavior is the same.
\end{lemma}
Note that from now on $\K(\q)$ denotes a ray class field of conductor $\q$, while we denote  for any $\mathbb{V}\neq \K$ the field constructed in Lemma \ref{field} by $\mathbb{V}(\q)$.
\begin{proof}
As $\mathfrak{q}$ is unramified in $\H_n/\K$ it follows that $\K(\mathfrak{q})\cap \H_n=\K(1)\cap\H_n=\K(1)$. Hence, $\Gal(\H_n\K(\mathfrak{q})/\H_n)=\Gal(\K(\mathfrak{q})/\K(1))\cong (\mathcal{O}/\mathfrak{q})^{\times}/\mathcal{O}^{\times}$. As $\vert \mathcal{O}^{\times} \vert =2$ and $N\mathfrak{q}\equiv 1\mod 2^{l+1}$ we can extract a cyclic extension of degree $2^l$ over $\H_n$. By definition $\mathfrak{q}$ is totally ramified in $\H_n(\mathfrak{q})/\H_n$ and the extension is unramified outside $\q$. The rest of the claim is an immediate consequence of the fact that $\q$ is unramified in $\L_n$.
\end{proof}
If $\mathfrak{r}=\prod\mathfrak{q}_i$ with $\mathfrak{q}_i$ distinct primes in $S_{n,l}$ then we define $\V(\mathfrak{r})$ as the compositum of the $\V(\mathfrak{q}_i)$.

Having this in place we can define  Euler systems.
\begin{definition}
An Euler system is a set of global elements 
\[\{\alpha^{\sigma}(n,\mathfrak{r})\mid n\ge 0,\mathfrak{r}\in S_{n,l},\sigma\in Gal(\K(\f\p^2)/\K)\}\]
satisfying
\begin{itemize}
\item[i)] $\alpha^{\sigma}(n,\mathfrak{r})\in \L_n(\mathfrak{r})^{\times}$ is a global unit in $\L_n(\mathfrak{r})$ for $\mathfrak{r}\neq (1)$.
\item[ii)] If $\mathfrak{q}$ is a prime such that $\mathfrak{qr}\in S_{n,l}$ then
\[N_{\L_n(\mathfrak{rq})/\L_n(\mathfrak{r})}(\alpha^{\sigma}(n,\mathfrak{rq}))=\alpha^{\sigma}(n,\mathfrak{r})^{\textup{Frob}_{\mathfrak{q}}-1}\]
\item[iii)] $\alpha^{\sigma}(n,\mathfrak{rq})\equiv \alpha^{\sigma}(n,\mathfrak{r})^{(N\q-1)/2^{l}} \mod \lambda$ for every prime $\lambda$ above $\mathfrak{q}$.
\end{itemize}\end{definition}
Note that if we fix $\sigma$ and $n$ and let only $\mathfrak{r}$ vary we obtain an Euler system in the sense of \cite{Rubin} for the field $\L_n$. So in Rubin's language our Euler-System is a system of Euler systems indexed by the pairs $(\sigma, n)$. 

 We now give a precise definition of the elliptic units.
 \begin{definition}
 Let $\g\mid\f$ be a non-trivial ideal. We define the elliptic units $C_{\g,n}$ in $\L_n$ as the group of units (they are units by \cite[Chapter II Proposition 2.4 iii)]{Shalit}) generated by all the $\xi_{\alpha,\sigma,Q_{\g}}(P^{\sigma}_{n+2})$, where $Q_{\g}$ is aprimitive $\g$ division point and $P^{\sigma}_n$ is a $\p^n$-torsion point on $E^{\sigma}$. If $\g=(1)$ we define $C_{(1),n}$ as the group generated by all the units of the form $\prod_{i=1}^s\xi_{\alpha_i,\sigma}( P^{\sigma}_{n+2})^{m_i}$ with $\sum_{i=1}^s m_i(N\alpha_i-1)=0$ (they are units by \cite[Chapter II Exercise 2.4]{Shalit}). We define further the groups $C_{\g}=\lim_{\infty\leftarrow n}C_{\g,n}$ and the group $C(\g)=\prod_{\mathfrak{h}\mid\g}C_{\mathfrak{h}}$. We will also use the notation $C_n$ and $C$ instead of $C_n(\g)$ and $C(\g)$ if the conductor is clear form the context.
 \end{definition}
This allows us to prove the following Lemma.
\begin{lemma}
\label{euler}
For every $u\in C_{\g}$ there exists an Euler system such that $\alpha^{\sigma}(n,1)=u$. 
\end{lemma}
\begin{proof}
As the properties defining an Euler-system are multiplicative it suffices to consider the case of $u$ being one of the generators, i.e. $u=\xi_{\alpha,\sigma}(P_{n+2}^{\sigma}+Q_{\g})$. Assume first that $\g\neq (1)$ and let $\V_n=\K(\g\p^{n+2})$. Define \[\alpha^{\sigma}(n,\mathfrak{r})=N_{\K(\g\p^{n+2}\mathfrak{r})/\V_n(\mathfrak{r})}\xi_{\alpha,\sigma} (P^{\sigma}_{n+2}+Q_{\g}+\sum_{\mathfrak{l}\mid\mathfrak{r }}Q_{\mathfrak{l}}),\] where $\mathfrak{l}$ are primes in $S_{l,n}$. Then $\alpha^{\sigma}(n,1)=u$. It remains to show that $\alpha$ generates an Euler system.
 Using that $\sigma_{\q}=1$ and that 
 $\Gal(\L_n(\mathfrak{rq})/\L_n(\mathfrak{r}))=Gal(\V_n(\mathfrak{rq})/\V_n(\mathfrak{r}))$ we obtain:
\begin{align*}
N_{\L_n(\mathfrak{rq})/\L_n(\mathfrak{r})}(\alpha^{\sigma}(n,\mathfrak{rq}))&=N_{\V_n(\mathfrak{rq})/\V_n(\mathfrak{r})}N_{\K(\g\p^{n+2}\mathfrak{rq})/\V_n(\mathfrak{rq})}\xi_{\alpha,\sigma} (P^{\sigma}_{n+2}+Q_{\g}+\sum_{\mathfrak{l}\mid\mathfrak{rq }}Q_{\mathfrak{l}})\\
&=N_{\K(\g\p^{n+2}\mathfrak{r})/\V_n(\mathfrak{r})}N_{\K(\g\p^{n+2}\mathfrak{rq})/\K(\g\p^{n+2}\mathfrak{r})}\xi_{\alpha,\sigma} (P^{\sigma}_{n+2}+Q_{\g}+\sum_{\mathfrak{l}\mid\mathfrak{rq }}Q_{\mathfrak{l}})\\
&=N_{\K(\g\p^{n+2}\mathfrak{r})/\V_n(\mathfrak{r})}\xi_{\alpha,\sigma\sigma_{\q}}(\eta_{\sigma}(\mathfrak{q})(P^{\sigma}_{n+2}+Q_{\g}+\sum_{\mathfrak{l}\mid\mathfrak{rq}}Q_{\mathfrak{l}}))^{1-\textup{Frob}_{\mathfrak{q}}^{-1}}\\
&=N_{\K(\g\p^{n+2}\mathfrak{r})/\V_n(\mathfrak{r})}\xi_{\alpha,\sigma}(P^{\sigma}_{n+2}+Q_{\g}+\sum_{\mathfrak{l}\mid\mathfrak{r}}Q_{\mathfrak{l}})^{\textup{Frob}_{\mathfrak{q}}(1-\textup{Frob}_{\mathfrak{q}}^{-1})}\\
&=(\alpha^{\sigma}(n,\mathfrak{r}))^{\textup{Frob}_{\q}-1}
\end{align*}
It remains to check property iii): The group $\Gal(\K(\g\p^{n+2}\mathfrak{rq})/\V_n(\mathfrak{rq})\K(\g\p^{n+2}\mathfrak{r}))$ acts only on the $\q$-torsion points. By definition we obtain that \[\vert\Gal(\K(\g\p^{n+2}\mathfrak{rq})/\V_n(\mathfrak{rq})\K(\g\p^{n+2}\mathfrak{r}))\vert =(N\q-1)/2^{l}\] due to the fact that $\K(\p^{n+2}\g)\neq\K$ is non-trivial. Using the fact that $\q$-torsion points reduce to zero modulo $\lambda$ and that $\Gal(\K(\g\p^{n+2}\mathfrak{rq})/\V_n(\mathfrak{rq}))$ restricts surjectively to $\Gal(\K(\g\p^{n+2}\mathfrak{r})/\V_n(\mathfrak{r}))$ the claim is an easy consequence of the definitions.

If $\g=(1)$ 
we choose $\alpha^{\sigma}(n,\mathfrak{r})=\prod_{i=1}^s\xi_{\alpha_i,\sigma}(P^{\sigma}_{n+2}+\sum_{\mathfrak{l}\mid\mathfrak{r}}Q_{\mathfrak{l}})^{m_i}$ and proceed as above. 
\end{proof}

For every prime $\q\in S_{n,l}$ we fix a generator $\tau_{\q}$ of $G_{\q}=\Gal(\L_n(\q)/\L_n)$ and define the following group ring elements
\[N_{\q}=\sum_{i=0}^{2^l-1}\tau_{\q}^i\quad D_{\q}=\sum_{i=0}^{2^l-1}i\tau_{\q}^i.\]
For $\mathfrak{r}=\prod_{k=1}^s\q_k$ we define $D_{\mathfrak{r}}=\sum_{k=1}^s D_{\q_k}\in \Z[\Gal(\L_n(\mathfrak{r})/\L_n)]$. 

With these definitions we have the following lemma .
\begin{lemma}
\label{kappa}\cite[Proposition 2.2]{Rubin}
For every Euler system $\alpha^{\sigma}(n,\mathfrak{r})$ there exists a canonical map \[\kappa\colon S_{n,l}\to \L_n^{\times}/(\L_n^{\times})^{2^l}\] such that $\kappa(\mathfrak{r})=(\alpha^{\sigma}(n,\mathfrak{r}))^{D_{\mathfrak{r}}}\mod (\L_n(\mathfrak{r}))^{2^l}$. 
\end{lemma}
 
For every prime ideal $\q\in S_{n,l}$ of $\K$ we define the free group of ideals in $\L_n$ \[I_{\q}=\oplus_{\mathfrak{Q}\mid\q}\Z\mathfrak{Q}=\Z[\Gal(\L_n/\K)]\mathfrak{Q}.\]
For every $y\in \L_n^{\times}$ denote by $[y]_{\q}$ the coset of the principal ideal $(y)$ in $I_{\q}/2^lI_{\q}$.
Let $\tilde{\mathfrak{Q}}$ be a prime above $\mathfrak{Q}$ in $\L_n(\q)$ and note that for every $x\in \L_n(\q)^{\times}$ the element $x^{1-\tau_{\q}}$ lies in $(\mathcal{O}(\L_n(\q))/\tilde{\mathfrak{Q}})^{\times}$. As $\mathcal{O}(\L_n(\q))/\tilde{\mathfrak{Q}}\cong \mathcal{O}(\L_n)/\mathfrak{Q}$ there is a well defined image $\overline{x^{1-\tau_{\q}}}$ in $(\mathcal{O}(\L_n)/\mathfrak{Q})^{\times}$. Thus, we can define a map
\[\L_n(\q)^{\times}\to(\mathcal{O}(\L_n)/\mathfrak{Q})^{\times}/((\mathcal{O}(\L_n)/\mathfrak{Q})^{\times})^{2^l}\quad x\mapsto (\overline{x^{1-\tau_{\q}}})^{1/d},\]
where $d=(N\q-1)/2^l$.
This map is surjective and the kernel of this map consists precisely of the elements whose $\tilde{\mathfrak{Q}}$ valuation is divisible  by ${2^l}$. Let now $w\in (\mathcal{O}(\L_n)/\mathfrak{Q})^{\times}/((\mathcal{O}(\L_n)/\mathfrak{Q})^{\times})^{2^l}$ and let $x$ be a preimage. Define 
\[l_{\mathfrak{Q}}(w)=\ord_{\tilde{\mathfrak{Q}}}(x)\mod 2^l\in \Z/2^l\Z.\] Note that $l_{\mathfrak{Q}}$ is a well defined isomorphism. Thus, we can define
\[\varphi_{\q}:(\mathcal{O}(\L_n)/\q)^{\times}/((\mathcal{O}(\L_n)/\q)^{\times})^{2^l}\to I_{\q}/I_{\q}^{2^l} \quad w\mapsto \sum_{\mathfrak{Q}\mid \q}l_{\mathfrak{Q}}(w)\mathfrak{Q}.\]
With these notations we have the following proposition. 
\begin{proposition}\cite[Proposition 2.4]{Rubin}
\label{kappa-properties}
Let $\alpha^{\sigma}(n,\mathfrak{r})$ be an Euler system and $\kappa$ be the map defined in Lemma \ref{kappa}. Let $\mathfrak{r}$ be a prime in $S_{n,l}$ and $\mathfrak{q}$ be a prime in $\K$. Then
\begin{itemize}
\item[i)] If $q\nmid \mathfrak{r}$ then $[\kappa(\mathfrak{r})]_{\mathfrak{q}}=0$.
\item[ii)] Assume that $\q\mid \mathfrak{r}$ and $\mathfrak{r/q}\neq (1)$. Then $[\kappa(\mathfrak{r})]_{\q}=\varphi_{\q}(\kappa(\mathfrak{r}/\q))$
\item[iii)] Assume that $\mathfrak{r}=\mathfrak{q}$ and that the ${\mathfrak{Q}}$-valuation of $(\alpha^{\sigma}(n,(1)))$ is divisible by $2^l$ for all $\mathfrak{Q}$ above $\q$ in $\L_n$. Then $[\kappa(\mathfrak{r})]_{\q}=\varphi_{\q}(\kappa(\mathfrak{r}/\q))$.
\end{itemize}
Note that Rubin does not distinguish between the cases ii) and iii). But as Bley \cite[Proposition 3.3]{BL} points out, the extra assumption in iii) is necessary.\end{proposition}
Let $y$ be any element in the kernel of $[\cdot]$. Then $y=\mathfrak{B}^{2^l}\mathfrak{C}$, where $\mathfrak{B}$ is an ideal only divisible by primes above $\q$ and $\mathfrak{C}$ is corpime to $\q$. Let $(\beta)=\mathfrak{B}\mathfrak{D}$ for some ideal $\mathfrak{D}$ coprime to $\q$. Then $y=\beta^{2^l}u$ and $u$ is coprime to $\q$. 
In particular, $u$ is a unit  at all ideals above $\q$. Thus, $\varphi_{\q}(u)$ is well defined and we can extend $\varphi_{\q}$ on $\ker([\cdot]_{\q})$.

\subsection{An Application of Tchebotarev's Theorem}
This section follows ideas of Bley in \cite{BL} and of Greither in \cite{greither}. As some steps of the proofs are slightly different for the case $p=2$ we will carry them out in detail.
The main goal of this section is to prove the following Theorem.
\begin{theorem}
\label{tchebotarev}
Let $\M=\L_n$ for some $n$ and write $G=\Gal(\M/\K)$. Assume that $\overline{\p}^c$ is the precise power of $\overline{\p}$ dividing the conductor of the extension $\M/\K$. Let $M=2^l$ for some $l$ and let $W\subset \M^{\times}/(\M^{\times})^M$ be a finite $\Z[G]$-module. Assume that there is a $\Z[G]$- homomorphism $\psi\colon W\to \Z/M\Z[G]$. Let $C\in A(\M)$ be an arbitrary ideal class. Then there are infinitely many primes $\mathfrak{Q}$ in $\M$ satisfying:
\begin{itemize}
\item[i)]$[\mathfrak{Q}]=2^{3c+4}C$.
\item[ii)] If $\mathfrak{q}=\mathfrak{Q}\cap\K$ then $N\mathfrak{q}\equiv 1\mod 2M$ and $\mathfrak{q}$ is totally split in $\M$.
\item[iii)] For all $w\in W$ one has $[w]_{\mathfrak{q}}=0$ and there exists a unit $u$ in $\Z/M\Z$ such that $\varphi_{\mathfrak{q}}(w)=2^{3c+4}u\psi(w)\mathfrak{Q}$. 
\end{itemize} 
 \end{theorem}
 A similar result was also proved by Vigui\'{e} in \cite{vg2} including the case $p=2$. As our result has slighly different assumptions and to carry out the differences for the case $p=2$ we give a complete proof here. 
 The proof of Theorem \ref{tchebotarev} relies on several lemmas which we will prove in the following. We fix the following Notation: Let $\H$ be the Hilbert class field of $\M$ and define $\M'=\M(\zeta_{{2M}})$ and $\M''=\M'(W^{1/M})$. 
 \begin{lemma}
 \label{claima}
 $[\H\cap\M':\M]\le 2^{c-1}$ if $c\ge 1$. The extension $\H\cap \M'/\M$ is trivial if $c=0$.
 \end{lemma}
 \begin{proof}
 As $2$ is totally split in $\K/\Q$ the ideal $\overline{\p}$ is totally ramified in $\K(\zeta_{{2M}})/\K$ and the ramification index is $M$. If $c=0$ then $\M/\K$ is unramified at $\overline{\p}$ and $\M'/\M$ is totally ramifed at all primes above $\overline{\p}$. Hence, $\M'\cap \H=\M$ and the claim follows in this case. Assume now that $c\ge 1$, then the ramification index of $\overline{\p}$ in $\M/\K$ is at most $\vert (\mathcal{O}(\K)/\overline{\p}^c)^{\times}\vert$ . Hence, the ramification index of every divisor of $\overline{\p}$ in $\M'/\M$ is at least $M/2^{c-1}$. In particular, $[\M':\M'\cap\H]\ge M/2^{c-1}$. Using that $[\M':\M]\le M$ it follows that $[\H\cap \M':\M]\le 2^{c-1}$. 
 \end{proof}
 \begin{lemma}
 \label{claimb}
 If $c=0$ then the group $\Gal(\M''\cap\H/\M)$ is annihilated by $4$. If $c>1$ then $\Gal(\M''\cap\H/\M)$ is annihilated by $2^{2c}$. In both cases it is annihilated by $2^{2c+2}$.
 \end{lemma}
 \begin{proof}
 By definition we have $[\K(\zeta_{2M}):\M\cap\K(\zeta_{2M})]\ge \textup{min}(M,M/2^{c-1})$. Consider first the case $c\ge 1$. As $\Gal(\K(\zeta_{2M})/\K)\cong \Z/2\Z\times \Z/(M/2)\Z$ we can choose an element $j$ in $\Gal(\K(\zeta_{2M})/\M\cap\K(\zeta_{2M}))$ of order $M/2^c$. Choose $r\in \Z$ such that $j(\zeta_{2M})=\zeta_{2M}^r$. It follows that $r^{M/2^c}\equiv 1\mod 2M$ and $r^b\not\equiv 1 \mod 2M$ for every $0<b<M/2^c$. The element $j$ has a lift to $\Gal(\M'/\M)$ of the same order. Let $\sigma$ be in $\Gal(\M''/\M')$ and $\alpha$ in $\M''$ such that $\alpha^{M}=w$ generates a maximal cyclic subgroup of $W$. By Kummer Theory there exists an even integer $t_w$ such that $\sigma(\alpha)=\zeta_{2M}^{t_w}\alpha$. 
Using the equivariant Kummer-pairing one can show 
 \begin{equation}
 \label{konjugation}
 j\sigma j^{-1}=\sigma^{r}.\end{equation}
 The extension $(\M''\cap\H\M')/\M$ is clearly abelian. Hence $\Gal(\M'/\M)$ acts trivially on the group $H=\Gal(\M''\cap\H\M'/\M')$. Together with \eqref{konjugation} this implies that $H$ is annihilated by $r-1$. On the other hand it is a Kummer extension of exponent at most $M$. Therefore, $H$ is annihilated by $2^d=\gcd(M,r-1)$. Then $r\equiv 1\mod 2^d$. Assume now that $r^{2^{v-d}}\equiv 1\mod 2^v$ for some $v\ge d$. Then $r^{2^{v+1-d}}\equiv 1\mod 2^{v+1}$. This shows that $r^{2^{l+1-d}}\equiv 1\mod 2^{l+1}$. Recall that $M=2^l$ and that $r^b\not\equiv 1\mod 2M$ for all $0<b<M/2^c$. It follows that $M/2^c\mid 2M/2^d$ and $c\ge d-1$. Therefore $2^{c+1}$ annihilates $H$. There is a natural surjective projection \[H\to \Gal(\M''\cap \H/\M'\cap \H).\] Using Lemma \ref{claima} this gives the claim in the case $c\neq 0$.
 
 \smallskip In the case $c=0$ we choose $j$ of order $M/2$. Then $r^{M/2}\equiv 1\mod 2M$ and $r^b\not\equiv 1 \mod 2M$ for all $0<b<M/2$. Using the same aguments as for the case $c=1$ we obtain that the extension $\M''\cap\H\M'/\M'$ is annihilated by $4$. This implies the claim in the case $c=0$.
 \end{proof}
 Using the Kummer pairing we see that there is a homomorphism \[F:\Gal(\M''/\M')\to \textup{Hom}(W,\zeta_{M})\] given by $F(\sigma)(w)=\sigma(w^{1/M})/w^{1/M}$. 
 \begin{lemma}
 \label{cohomology}
 $2^{c+2}$ annihilates the cokernel of $F$.
 \end{lemma}
 \begin{proof}
 Let $W'$ be the image of $W$ in $\M'/(\M')^{M}$. By Kummer duality  we have $\textup{Hom}(W',\langle\zeta_M\rangle)\cong \Gal(\M''/\M')$. Let $f\colon (\M^{\times})/(\M^{\times})^M\to (\M'^{\times})/(\M'^{\times})^M$ be the natural map. Using the exact sequence
 \[0\to \ker(f)\to W\to W'\to 0\] we obtain a second exact sequence \[\textup{Hom}(W',\langle\zeta_M\rangle)\to\textup{Hom}(W,\langle\zeta_M\rangle)\to\textup{Hom}(\ker(f),\langle\zeta_M\rangle).\] Hence, to prove the lemma it suffices to prove that the kernel of $f$ is annihilated by $2^{c+2}$. Let $u\in \Ker(f)$ and choose an element $v\in\M'$ such that $u=v^M$. We define $\delta_v:\Gal(\M'/\M)\to \langle \zeta_M\rangle$ by $\delta_v(g)=g(v)/v$. As \[\delta_v(gh)=gh(v)/g(v)\cdot g(v)/v=\delta_v(g)\cdot g\delta_v(h).\] It follows that $\delta_v$ is a cocycle. Note that $v$ is unique up to $M$-th roots of unity. If we cchoose $v'=v\zeta^c_M$ we obtain $\delta_{v'}(g)=g(v)/v\cdot g(\zeta_M^c)/\zeta_M^c$. Hence, $\delta_v$ is uniquely defined up to coboundaries and $\delta_v$ has a well defined image in $H^1(\Gal(\M'/\M),\langle\zeta_M\rangle)$. Thus, we have an injective map $\ker(f)\hookrightarrow H^1(\Gal(\M'/\M),\langle\zeta_M\rangle)$. Therefore it suffices to bound $ H^1(\Gal(\M'/\M),\langle\zeta_M\rangle)$. If the group $\Gal(\M'/\M)$ is cyclic we see that $\langle\zeta_M\rangle$ has a trivial Herbrandt quotient. So it suffices to consider \[\vert H^0(\Gal(\M'/\M),\langle \zeta_M\rangle)\vert\le \vert \langle\zeta_M\rangle\cap \M\vert\le 2^{c+1}.\]
 If $\Gal(\M''/\M)$ is not cyclic then it is isomorphic to $\Delta\times C_r$ where $C_r$ is cyclic and $\Delta\cong \Z/2\Z$. Using the exact sequence
 \[H^1(\Delta,\langle \zeta_4\rangle)\to H^1(\Gal(\M'/\M),\langle\zeta_M\rangle)\to H^1(C_r,\langle\zeta_M\rangle)\]
 and the fact that the last term is annihilated by $2^{c+1}$ while the first one is annihilted by $2$ we obtain that the middle term is annihilated by $2^{c+2}$ proving the lemma.
 \end{proof}

 Now we have all ingredients to prove Theorem \ref{tchebotarev}. 
 \begin{proof}[Proof of Theorem \ref{tchebotarev}]
 Consider the map $\iota:(\Z/M\Z)[G]\to \langle \zeta_M\rangle$ defined by $\sum a_{\sigma}\sigma\to \zeta_M^{a_1}$. Then $\iota\circ \psi\in \textup{Hom}(W,\langle\zeta_M\rangle)$. Using Lemma \ref{cohomology} we see that $2^{c+2}(\iota\circ \psi)$ has a preimage $\gamma$ in $
 \Gal(\M''/\M')$. Let $\gamma_1=2^{c+2}\left(\frac{C}{\H/\M}\right)$ and choose $\delta\in \Gal(\M''\H/\M)$ such that $\delta\mid_{\H}=2^{2c+2}\gamma_1$ and $\delta\mid_{\M''}=2^{2c+2}\gamma$. Note that this is possible as $\Gal(\M''\cap\H/\M)$ is annihilated by $2^{2c+2}$ due to Lemma \ref{claimb}. Using Tchebotarev's Theorem we can find infinitely many primes $\mathfrak{Q}\in \M$ of degree $1$ such that \[\left(\frac{\mathfrak{Q}}{\M''\H/\M}\right)=\textup{conjugacy class of $\delta$}.\]
 As $\delta\mid_{\M'}=2^{2c+2}\gamma\mid_{\M'}=\textup{id}$ we see that $\mathfrak{Q}$ is totally split in $\M'/\M$. Let $\q=\mathfrak{Q}\cap \K$. Then $\q$ is totally split in $\M'/\K$ and $N\q\equiv 1\mod 2M$.
 Further $\delta\mid_{\H}=2^{2c+2}\gamma_1\mid_{\H}=2^{3c+4}\left(\frac{C}{\H/\M}\right)$. It follows that $[\mathfrak{Q}]=2^{3c+4}C$.
 It remains to prove point iii) of the Theorem. To do so we note that \[\ord_{\mathfrak{Q}}(2^{3c+4}\psi(w)\mathfrak{Q})\equiv 0\mod M
 \Leftrightarrow 2^{3c+4}\iota\circ \psi(w)=1.\] Using that $\gamma$ is the preimage of $2^{c+2}\iota\circ \psi$ we see that \[2^{3c+4}\iota\circ \psi(w)=1 \Leftrightarrow (2^{2c+2}\gamma)w^{1/M}/w^{1/M}=1.\] As $\mathfrak{Q}$ has Artin-symbol $2^{2c+2}\gamma$ in $\M''/\M$ we see that 
 \[\ord_{\mathfrak{Q}}(2^{3c+4}\psi(w)\mathfrak{Q})\equiv 0\mod M\Leftrightarrow  w \textup{ is an } M\textup{-th power modulo }\mathfrak{Q}.\]  $w$ is an $M$-th power in $\M''$ and $\mathfrak{Q}$ is not ramified in $\M''/\M$. Therefore, $[(w)]_{\q}=0$. By definition $\varphi_{\q}(w)=0\Leftrightarrow w \textup{ is an } M \textup{-th power modulo } \mathfrak{Q}$. It follows that $\ord_{\mathfrak{Q}}(2^{3c+4}\psi(w)\mathfrak{Q})=u'\ord_{\mathfrak{Q}}(\varphi_{\q}(w))$ for some unit $u'$. From this the claim follows as in \cite[page 403]{lang}.
 \end{proof}
 
\subsection{$\chi$-components on the class group and on $E/C$} 
\label{sec;components}
Recall that we fixed a decomposition $\Gal(\L_{\infty}/\K)\cong \Gamma'\times H$ with $H=\Gal(\L_{\infty}/\K_{\infty})$. Let $\gamma'$ be a topological generator of $\Gamma'$. To simplify notation we will use the notation $\gamma'_n$ for the element $\gamma'^{2^n}$. Let $\Gamma=\Gal(\L_{\infty}/\L)$. There exists a power of $2$ such that $\Gamma'^{2^m}$ is contained in $ \Gal(\L_{\infty}/\K(\f\p^2))$. In particular $\L_{\infty}/\L_{\infty}^{\Gamma'^{2^m}}$ is totally ramified at all primes above $\p$ and ${\Gamma'^{2^{m+n}}}={\Gamma^{2^{m'+n}}}$ for some $m'\le m$ independent of $n$. Recall that $A_n$ denotes the class group of $\L_n$, i. e. $\gamma'^{2^{m+n-m'}}$ acts trivial on $A_n$ and $\L_n$ for $n\ge m'$. We fixed the notations $\Lambda=\lim_{\infty\leftarrow  n}\Z_p[[\Gal(\L_n/\K)]]$ and $A_{\infty}=\lim_{\infty\leftarrow n}A_n$. Let $\chi$ be a character of $H$. Then $A_{\infty,\chi}$ and $(\overline{E}/\overline{C})_{\chi}$ are $\Lambda_{\chi}$-modules. Let $\Lambda_{\chi,n}$ be the quotient of $\Lambda_{\chi}$ by  $1-\gamma'_n$.
In particular, there is a pseudo-isomorphism
\begin{equation}\label{pseudo}
A_{\infty,\chi}\sim\bigoplus_{i=1}^k\Lambda_{\chi}/g_i,
\end{equation}
with finite kernel and cokernel.
\begin{lemma}
\label{leo}
The kernel of the multiplication by $(1-\gamma'_n)$ on $A_{\infty}$ is finite for every $n$.
\end{lemma}
\begin{proof}
This follows directly from the fact that all finite subextension of $\L_{\infty}/\K$ are abelian over $\K$ and that the Leopoldt conjecture holds for any abelian extensions of imaginary quadratic fields. In particular, Leopoldt's conjecture holds for every field $\L_n$ (see \cite{greither} for more details).
\end{proof}
\begin{lemma}
\label{pseudoA}
Let $\chi$ be a character of $H$ and $n\ge m'$. Then there is a $\Lambda_{\chi,n}$ homomorphism 
\[A_{\chi,n}\to\bigoplus_{i=1}^k\Lambda_{\chi,m+n-m'}/(\overline{g}_i),\]
with uniformly bounded cokernel.
Here, $\overline{g_i}$ is the restriction of $g_i$ to level $n$. 
\end{lemma}
\begin{proof}
This proof is very similar to \cite[Lemma 3.11]{greither}: By \cite[page 281]{washington} the module $A_n$ is isomorphic to $A_{\infty}/\nu_{m+n-m',m}Y$ for some submodule $Y$. Consider the map \[\phi_n\colon A_{\infty}/(1-\gamma'_{m+n-m'})A_{\infty}\to A_n.\] By definition the kernel is isomoprhic to $\nu_{m+n-m',m}Y/(1-\gamma'_{m+n-m'})A_{\infty}$ which is bounded by the size of $Y/(1-\gamma'_{m})A_{\infty}\le A_{\infty}/(1-\gamma'_m)A_{\infty}$. By Lemma \ref{leo}
this quotient is finite and the kernel of $\phi_n$ is uniformly bounded. Thus, the kernel of the natural projections \[A_{\infty,\chi}/(1-\gamma'_{m+n-m'})A_{\infty,\chi}\to A_{\chi,n}\] has a uniformly bounded kernel and we can deduce the claim from \eqref{pseudo}.
\end{proof}
Let $\Gamma'_{n_2,n_1}=\Gamma'^{2^{n_1}}/\Gamma'^{2^{n_2}}$ for $n_2>n_1$. Recall that $\Gamma'^{2^{n}}$ fixes the field $\L_{m'+n-m}$ for $n>m$. Hence $\Gal(\L_{n_2-m+m'}/\L_{n_1-m+m'})=\Gamma'_{n_2,n_1}$.
\begin{lemma}
There is a constant $k$ such that
\label{H0andh1}
 \[\vert (1-\gamma'_m)H^1(\Gamma'_{n_2,n_1},E_{m'+n_2-m})\vert \le 2^k\textup{ and } \vert (1-\gamma'_m)H^0(\Gamma'_{n_2,n_1},E_{m'+n_2-m})\vert \le 2^k\] for any pair $(n_1,n_2)$ with  $n_2>n_1\ge m$.
\end{lemma}
\begin{proof}
The proof follows the ideas of \cite[Lemma 1.2]{Rubin2}. But it is restated here as we use weaker assumptions.
Let $E'_{m'+n_2-m}$ be the $\p$-units in $\L_{m'+n_2-m}$ and $R_{m'+n_2-m}$ be the $\Z_p$-free group defined by the exact sequence
\[0\to E_{m'+n_2-m}\to E'_{m'+n_2-m}\to R_{m'+n_2-m}\to 0\]
As $\L_{\infty}/\L_{m'}$ is totally ramified we see that $\Gamma'^{m}$ acts trivially on $R_{m'+n_2-m}$. We know from \cite[page 267]{iwasawa} that $\vert H^1(\Gamma'_{n_2,n_1},E'_{m'+n_2-m})\vert$ is uniformly bounded\footnote {in Iwasawa's notation $E'$ are the $p$-units, but the proof of Theorem 12 works for the $\p$-units as well.}. Further, we have the exact sequence \[H^0(\Gamma'_{n_2,n_1},R_{m'+n_2-m})\to H^1(\Gamma'_{n_2,n_1},E_{m'+n_2-m})\to H^1(\Gamma'_{n_2,n_1},E'_{m'+n_2-m})\]
The first term is annihilated by $1-\gamma'_m$ and the last term is uniformly bounded. It follows that $(1-\gamma'_m)H^1(\Gamma'_{n_2,n_1},E_{m'+n_2-m})$ is uniformly bounded. 

\smallskip
It is an immediate consequence from \cite[V Theorem 2.5]{Janusz} that $q(E'_{m'+n_2-m})=2^{(n_2-n_1)(1-s)}$, where $s$ is the number of primes above $2$. Thus, 
\[2^{(n_2-n_1)(s-1)}\vert H^1(\Gamma'_{n_2,n_1},E'_{m'+n_2-m}) \vert =\vert H^0(\Gamma'_{n_2,n_1},E'_{m'+n_2-m})\vert.\]
Consider the map $H^0(\Gamma'_{n_2,n_1},E'_{m'+n_2-m})\to N_{n_1,m}E'_{m'+n_1-m}/N_{n_2,m}E'_{m'+n_2-m}$ induced by $N_{n_1,m}=(\gamma'_{n_1}-1)/(\gamma'_{m}-1)$. Using that $N_{n_1,m}(1-\gamma'_m)=(1-\gamma'_{n_1})$ and  that $\Gamma'^{2^{n_1}}$ is precisely the group fixing $\L_{m'+n_1-m}$ we see that the subgroup \[((1-\gamma'_m)E'_{m'+n_1-m}+N_{n_2,n_1}E'_{m'+n_2-m})/N_{n_2,n_1}E'_{m'+n_2-m}\] is certainly contained in the kernel. Note that $N_{n_1,m}E'_{m'+n_1-m}/N_{n_2,m}E'_{m'+n_2-m}$ is the kernel of the natural map $H^0(\Gamma'_{n_2,m},E'_{m'+n_2-m})\to H^0(\Gamma'_{n_1,m},E'_{m'+n_1-m})$. Thus, we obtain:
\begin{align*}
\vert (1-\gamma'_m)H^0(\Gamma'_{n_2,n_1},E'_{m'+n_2-m})\vert&\le \frac{\vert H^0(\Gamma'_{n_2,n_1},E'_{m'+n_2-m})\vert\vert H^0(\Gamma'_{n_1,m},E'_{m'+n_1-m}) \vert }{\vert H^0(\Gamma'_{n_2,m},E'_{m'+n_2-m})\vert }\\&\le \frac{2^{(n_2-n_1)(s-1)+k}2^{(n_1-m)(s-1)+k}}{2^{(n_2-m)(s-1)}}=2^{2k},
\end{align*}
where $2^k$ is the uniform bound on $H^1(\Gamma'_{n_2,n_1},E'_{m'+n_2-m})$. It is easy to verify that the natural map $H^0(\Gamma'_{n_2,n_1},E_{m'+n_2-m})\to H^0(\Gamma'_{n_2,n_1},E'_{m'+n_2-m})$ is an injection and the claim follows.
\end{proof}

\begin{lemma}
\label{kernel}
Let $n\ge m'$ and consider the projection \[\pi_n\colon\overline{E}_{\infty}/(1-\gamma'_{m+n-m'})\overline{E}_{\infty}\to \overline{E}_n.\] There exists an integer $k$ such that $2^k(1-\gamma'_m)$ annihilates the kernel and the cokernel of $\pi_n$ for all $n\ge m$.
\end{lemma}
\begin{proof}
We have an exact sequence
\begin{align*}
&\lim_{\infty\leftarrow n'}H^1(\Gamma'_{m+n'-m',m+n-m'},\overline{E}_{n'})\to \\\to&\overline{E}_{\infty}/(1-\gamma'_{m+n-m'})\overline{E}_{\infty}\to \overline{E}_n\to \lim_{\infty\leftarrow n'}H^0(\Gamma'_{m+n'-m',m+n-m'},\overline{E}_{n})\end{align*}
The first and the last term are annihilated by $2^k(1-\gamma'_m)$ du to Lemma \ref{H0andh1} and the claim follows.
\end{proof}
\begin{lemma}
\label{pseudocyclic} Let $U_{\infty}$ be defined as in the introduction. Then
$U_{\infty}\otimes_{\Z_p}\Q_p\cong\Lambda\otimes_{\Z_p}\Q_p$ and $U_{\infty,\chi}\otimes_{\Z_p}\Q_p\cong\Lambda_{\chi}\otimes_{\Z_p}\Q_p$.
\end{lemma}
\begin{proof}
The first claim follows as in \cite[Lemma 3.5 Claim 2]{BL}. Bley gives two references for this proof. Note that the second one is only stated for $p>2$ but the proof works for $p=2$ as well. 

The second claim can be proved as follows: \[U_{\infty}\otimes_{\Z_p}\Z_p(\chi)\otimes_{\Z_p}\Q_p\cong \Lambda\otimes_{\Z_p}\Z_p(\chi)\otimes_{\Z_p}\Q_p.\] Let $I_{\chi}\subset \Z(\chi)[H]$ be the module generated by $\sigma-\chi(\sigma)$ for $\sigma\in H$. It is an easy verification that \begin{align*}&U_{\infty}\otimes_{\Z_p}\Z_p(\chi)\otimes_{\Z_p}\Q_p/I_{\chi}(U_{\infty}\otimes_{\Z_p}\Z_p(\chi)\otimes_{\Z_p}\Q_p)\\&\cong \Lambda\otimes_{\Z_p}\Z_p(\chi)\otimes_{\Z_p}\Q_p/I_{\chi}(\Lambda\otimes_{\Z_p}\Z_p(\chi)\otimes_{\Z_p}\Q_p).\end{align*}
It is proved in \cite[Lemma 2.1]{Tsuji} that $M_{\chi}\cong (M\otimes_{\Z_p}\Z_p(\chi))/I_{\chi}(M\otimes_{\Z_p}\Z_p(\chi))$.
 Further, for any module $M$ we see that \begin{align*}&M\otimes_{Z_p}\Z_p(\chi)\otimes_{\Z_p}\Q_p/I_{\chi}(M\otimes_{Z_p}\Z_p(\chi)\otimes_{\Z_p}\Q_p)\\&=(M\otimes_{\Z_p}\Z_p(\chi)/I(\chi)(M\otimes_{\Z_p}\Z_p(\chi)))\otimes \Q_p=M_{\chi}\otimes_{\Z_p}\Q_p.\end{align*} Using this for $U_{\infty}$ and $\Lambda$ the second claim follows.
\end{proof}
\begin{lemma}
\label{thetamap}
Let $h_{\chi}$ be the characteristic ideal of $(\overline{E}/\overline{C})_{\chi}$.
Let $n\ge m$. Then there exist constants $n_0$, $c_1$ and $c_2$ indeependent of $n$, a divisor $h'_{\chi}$ of $h_{\chi}$ and a $\Gal(\L_{m'+n-m}/\K)$-homomorphism 
\[\vartheta_{m'+n-m}:\overline{E}_{m'+n-m,\chi}\to \Lambda_{n,\chi}\]
such that
\begin{itemize}
\item[i)] $h'_{\chi}$ is prime to $1-\gamma'_v$ for all $v$,
\item[ii)] $(\gamma'_{n_0}-1)^{c_1}2^{c_2}h'_{\chi}\Lambda_{n,\chi}\subset\vartheta_{m'+n-m} (im(\overline{C}_{m'+n-m,\chi}))$,
 wehere $im(\overline{C}_{m'+n-m,\chi})$ denotes the image of $\overline{C}_{m'+n-m,\chi}$ in $\overline{E}_{m'+n-m,\chi}$.
\end{itemize}
\end{lemma}
\begin{proof}
From the second claim of Lemma \ref{pseudocyclic} and the fact that $\Lambda_{\chi}\otimes_{\Z_p}\Q_p$ is a principal ideal domain we obtain that the submodule $\overline{E}_{\infty,\chi}$ is free cyclic over the ring $\Lambda_{\chi}\otimes_{\Z_p}\Q_p$. Consider the map \[\pi_n: \overline{E}/(1-\gamma'_n)\to \overline{E}_{m'+n-m}.\]
The rest of the proof is exactly the same as \cite[Lemma 3.5]{BL}; we just have  to substitute $E_n$ by $E_{m'+n-m}$ and $(1-\gamma')$ by $(1-\gamma'_m)$ in all computations due to the new definition of $\pi_n$  and the fact that Lemma \ref{kernel} is weaker than the corresponding claim in Bley's case. Note that Bley states the lemma only for non trivial characters but this assumption is not necessary.
\end{proof}
\begin{lemma}
\label{tecnical}
Let $\M=\L_n$ for some $n$ and let $\Delta$ be a subgroup of $\Gal(\M/\K)$. Let $\chi$ be a character of $\Delta$, $M=2^l$ and $\mathfrak{A}=\prod_{i=1}^s\mathfrak{q}_i\in S_{n,l}$. Let $\mathfrak{Q}$ be a divisor of $\q_s$ in $\M$ and let $C=[\mathfrak{Q}]$ the ideal class of $\mathfrak{Q}$. Let $B\subset A(\M)$ be the subgroup generated by ideals dividing $\mathfrak{A}/\q_s$. Let $x\in \M^{\times}/(\M^{\times})^M$ be such that $[x]_{\mathfrak{r}}=0$ for all $(\mathfrak{r},\mathfrak{A})=1$. Let $W\subset  \M^{\times}/(\M^{\times})^M$ be the $\Z_p[G]$-submodule generated by $x$. Assume that there are elements $E,\eta,g\in\Z_p[G]$ such that 
\begin{itemize}
\item[i)] $E\cdot\textup{ann}_{\Z_p[G]}(C_{\chi})\subset g\Z_p[G]_{\chi}$, where $C_{\chi}$ is the projection of $C$ to $(A/B)_{\chi}$.
\item[ii)] $\Z_p[G]_{\chi}/g(\Z_p[G])_{\chi}$ is finite.
\item[iii)] $M\ge \vert A_{\chi}(\M)\vert \vert\eta ((I_{\q_s}/MI_{\q_s})/[W]_{\q_s})\vert$.
\end{itemize}
Then there exists a $G$-homomorphism 
\[\psi\colon W_{\chi}\to (\Z/M\Z)[G]\]
such that $(g\psi(x)\mathfrak{Q})_{\chi}=(E\eta[x]_{\q})_{\chi}$ in $I_{\q}/M I_{\q}$.
\end{lemma}
\begin{proof}
This is \cite[Lemma 3.8]{BL}. The proof is the same as \cite[Lemma 3.12]{greither}.
\end{proof}
To prove the central theorem of this section we need the following lemma. 
\begin{lemma}
\label{lift}\cite[Lemma 3.13]{greither}
Let $\Delta$ be any finite group and $N$ a $\Z_p[\Delta]$-module. Let $\chi$ be a character of $\Delta$ and $n:N\to N_{\chi}$ the natural projection. Then there exists a $\Z_p[\Delta]$-homomorphism \[\varepsilon_{\chi}: N_{\chi}\to N\] such that $n\circ \varepsilon_{\chi}=\vert \Delta\vert$. 
\end{lemma}
 
 Let $\q$ be an element in $S_{n,l}$ and $\mathfrak{A}$ in $I_{\q}$. Then there is an element $v_{\mathfrak{Q}}(\mathfrak{A})$ in $\Z/2^l\Z[\Gal(\L_n/\K)]$ such that $\mathfrak{A}=v_{\mathfrak{Q}}(\mathfrak{A})\mathfrak{Q}$.
The following Theorem allows us to relate the characteristic ideal of $A_{\chi}$ to the one of $(\overline{E}/\overline{C})_{\chi}$. The proof follows the ideas of \cite{BL}.
\begin{theorem}
Let $\M=\L_n$ and $G=\Gal(\M/\K)$ for n large enough. Let  $\chi$ be a character of $H\subset \Gal(\M/\K)$. For $1\le i\le k$ let $C_i\in A_{\chi}(\M)$ be such that $t(C_i)=(0,0,\cdots,2^{c_3},0\cdots,0)$ in $\oplus_{i=1}^k \Lambda_{\chi,m+n-m'}/(\overline{g}_i)$ where $t$ is the map defined in Lemma \ref{pseudoA} and $2^{c_3}$ annihilates the cokernel. Let $C_{k+1}$ in $A_{\chi}$ be arbitrary. Let $d=3c+4$ where $c$ is defined in Theorem \ref{tchebotarev}. Then there are prime ideals $\mathfrak{Q}_i$ in $\M$ such that
\begin{itemize}
\item[i)] $[\mathfrak{Q}_i]=2^dC_i$
\item[ii)] $\q_i=\mathfrak{Q}_i\cap\K$ is in $S_{n,n}$
\item[iii)] one has \begin{align}\label{iiia}
&(v_{\mathfrak{Q}_1}(\kappa(\q_1))_{\chi}=u_1\vert H\vert (\gamma'_{n_0}-1)^{c_1}2^{d+c_2}h'_{\chi} \mod 2^n\\
\label{iiib}
&(g_{i-1}v_{\mathfrak{Q}_i}(\kappa(\q_1\q_2\dots\q_i))_{\chi}\\\nonumber&=u_i\vert H\vert(\gamma'_{n_0}-1)^{c_1^{i-1}}2^{2d+c_3}(v_{\mathfrak{Q}_{i-1}}(\kappa(\q_1\dots\q_{i-1}))_{\chi} \mod 2^n
\textup{ for } 2\le i\le k+1.\end{align}
\end{itemize}
\begin{proof}
By Lemma \ref{thetamap} there exists an element $\xi'$ in $im(\overline{C}_{n,\chi})$ with the property $\vartheta_n(\xi')=(1-\gamma'_{n_0})^{c_1}2^{c_2}h'_{\chi}$. By approximating $\xi$ with a global elliptic unit we can find $\xi\in C_n$ such that $\vartheta_n(\xi)=(1-\gamma'_{n_0})^{c_1}2^{c_2}h'_{\chi}\mod M\Lambda_{\chi,m+n-m'}$. We can apply Lemma \ref{euler} to find an Euler system $\alpha^{\sigma}(n,\mathfrak{r})$ such that $\alpha^{\sigma}(n,(1))=\xi$. Let $i=1$ and $C$ be a preimage of $C_i$ under the map $A_n\to A_{\chi,n}$. Choose  $M=2^n$ and $W=\mathcal{O}(\M)^{\times}/(\mathcal{O}(\M)^{\times})^M$. Consider
\[\psi: W\to \Z/M\Z[G]\quad x\mapsto(\varepsilon_{\chi}\circ \vartheta_n)(x^v),\] where $v$ is such that $x^v\in E_n$ for all $x$ and $\varepsilon_{\chi}$ is defined as in Lemma \ref{lift}. Then Theorem \ref{tchebotarev} implies that we can find an ideal $\mathfrak{Q}_1$ satisfying i) and ii). We know further from Theorem \ref{tchebotarev} that $\varphi_{\q}(w)=2^du\psi(w)\mathfrak{Q}_1$. As $\alpha^{\sigma}(n,(1))$ is a unit we can apply Proposition \ref{kappa-properties} and obtain
\begin{align*}
v_{\mathfrak{Q}_1}(\kappa(\q_1))\mathfrak{Q}_1&=[\kappa(\q_1)]_{\q_1}=\varphi_{\q_1}(\kappa(1))\\&=\varphi_{\q_1}(\xi)=2^du\psi(\xi)\mathfrak{Q}_1=2^duv\varepsilon_{\chi}((\gamma'_{n_0}-1)^{c_1}2^{c_2}h'_{\chi})\mathfrak{Q}_1\mod 2^n.
\end{align*}
Projecting to the $\chi$ component and using the definition of $\varepsilon_{\chi}$ we get \eqref{iiia}.

We will now define the ideals $\mathfrak{Q}_i$ inductively. Assume that we have already found the $\mathfrak{Q}_1,\dots\,\mathfrak{Q}_{i-1}$ and let $\mathfrak{a}_{i-1}=\prod_{j=1}^{i-1}\q_i$. Using point iii) recursively we see
\[\prod_{j\le i-2}g_j(v_{\mathfrak{Q}_{i-1}}(\kappa(\mathfrak{a}_{i-1}))_{\chi}=\vert H\vert ^{i-1}u2^{(i-2)(2d+c_3)+d+c_2}(\gamma'_{n_0}-1)^{c_1+\sum_{j=1}^{i-2}c_1^j}h'_{\chi}.\]
Let $D_i=\vert H\vert ^{i-1}u2^{(i-2)(2d+c_3)+d+c_2}$. By enlarging $c_1$ we can assume that $c_1+\sum_{j=1}^{i-2}c_1^j$ is bounded by $c_1^{i-1}$ and set $t_i=c_1^{i-1}$. It follows that $v_{\mathfrak{Q}_{i-1}}(\kappa(\mathfrak{a}_{i-1}))_{\chi}\mid D_i h'_{\chi}(\gamma'_{n_0}-1)^{t_i}$. Define $N=(\gamma'_{n_0}-1)^{t_i}(I_{\q_{i-1}}/(MI_{\q_{i-1}}+\Z_p[G][\kappa(\mathfrak{a}_{i-1})]_{\q_{i-1}}))$. As $h'_{\chi}$ is coprime to every $\gamma'_n-1$ we see that $\Lambda_{\chi,m+n-m'}/h'_{\chi}$ is finite and further
\[\vert N\vert \le \vert \Lambda_{\chi,m+n-m'}/D_i\vert\vert \Lambda_{\chi,m+n-m'}/h'_{\chi}\vert.\] Choose now $2^l=M> \max(\vert A_{\chi}(\M)\vert \vert \Lambda_{\chi,m+n-m'}/D_i\vert\vert \Lambda_{\chi,m+n-m'}/h'_{\chi}\vert,2^n)$. We want to apply Lemma \ref{tecnical} with $E=2^{c_3+d}$, $\eta=(\gamma'_{n_0}-1)^{t_i}$, $g=g_{i-1}$, $\mathfrak{A}=\mathfrak{a}_{i-1}$,  and $x=\kappa(\mathfrak{a}_{i-1})$. To do so we have to check the assumptions. It follows directly from Proposition \ref{kappa-properties} i) that $[x]_{\mathfrak{r}}=0$ for all $\mathfrak{r}$ coprime to $\mathfrak{a}_{i-1}$. We now have to check the conditions i)-iii) from Lemma \ref{tecnical}.
\begin{itemize}
\item[i)] By definition $C=[\mathfrak{Q}_{i-1}]=2^dC_{i-1}$. The annihilator of  $t(C)$ is $g_{i-1}/(2^{c_3+d},g_{i-1})$ and we obtain that $E\cdot \textup{ann}_{\Z_p[G]}(C_{\chi})\subset g_{i-1}\Z_p[G]_{\chi}$.
\item[ii)] It is immediate from Lemma \ref{pseudoA} that $\Z_p[G]_{\chi}/\Z_p[G]_{\chi}g$ is finite.
\item[iii)] $M>\vert A_{\chi}\vert\vert N\vert=\vert A_{\chi}\vert \vert \eta( \frac{I_{\q_{i-1}}/(MI_{\q_{i-1}})}{\Z_p[G][\kappa(\mathfrak{a}_{i-1})]_{\q_{i-1}}})\vert$. 
\end{itemize}
Thus, we obtain a homomorphism
\[\psi_i\colon W_{\chi}\to\Z/M\Z[G]\] with  $g_{i-1}\psi_i(\kappa(\mathfrak{a}_{i-1}))_{\chi}=(2^{c_3+d}(\gamma'_{n_0}-1)^{t_i}v_{\mathfrak{Q}_{i-1}}(\kappa(\mathfrak{a}_{i-1})))_{\chi}$. Let $\Pi_{\chi}$ be the projection $W\to W_{\chi}$ and define $\psi=\varepsilon_{\chi}\circ\psi_i\circ \Pi_{\chi}$. Let $M$ be as in the previous paragraph and $C$ a preimage of $C_i$. Then Theorem \ref{tchebotarev} gives us a prime ideal $\mathfrak{Q}_i$ satisfying i) and ii) (recall that $2^n\mid M$). Further, $\varphi_{\q_i}(\kappa(\mathfrak{a}_{i-1}))=2^du\psi(\kappa(\mathfrak{a}_{i-1}))\mathfrak{Q}_i$. Then we obtain 
\begin{align*}
v_{\mathfrak{Q}_i}(\kappa(\q_1\dots\q_i))\mathfrak{Q}_i=[\kappa(\q_1\dots\q_i)]_{\q_i}=\varphi_{\q_i}(\kappa(\q_1\dots\q_{i-1}))\\=2^du\psi(\kappa(\mathfrak{a}_{i-1}))\mathfrak{Q}_i
\end{align*}
Projecting to the $\chi$-component and using the definition of $\psi_i$ we obtain
\[(g_{i-1}v_{\mathfrak{Q}_i}(\kappa(\q_1\q_2\dots\q_i))_{\chi}=u_i\vert H\vert(\gamma'_{n_0}-1)^{c_1^{i-1}}2^{2d+c_3}(v_{\mathfrak{Q}_{i-1}}(\kappa(\q_1\dots\q_{i-1}))_{\chi} \]
which finishes the proof.
\end{proof}
\end{theorem}

To derive a relation between $h_{\chi}$ and $\prod_{i=1}^s\g_i$ we need the following result which is proved in \cite[Theorem 1]{mu} and \cite[Theorem 1.1]{OV}. In the case of $\L$ being the Hilbert class field and $\K=\Q(\sqrt{-q})$ with $q$ a prime congruent to $7$ modulo $8$  this is \cite[Theorem 1.1]{ChoiKeLi}.
\begin{theorem}
\label{mu1}
Let $\M/\K$ be an arbitrary abelian extension and $\Omega/\K_{\infty}\M$ be the maximal $p$-abelian $\p$-ramified extension of $\M\K_{\infty}$ then $Gal(\Omega/\K_{\infty}\M)$ is finitely generated as $\Z_p$-module.
\end{theorem}
\begin{corollary}
\label{mu_2}
Let $\M$ be as above and consider $\H/\K_{\infty}\M$ (the maximal $p$-abelian unramified extension of $\K_{\infty}\M$). Then $A_{\infty}=\Gal(\H/\K_{\infty}\M)$ is finitely generated as a $\Z_p$-module
\end{corollary}
\begin{proof}
$\Gal(\H/\K_{\infty}\M)$ is a quotient of $Gal(\Omega/\K_{\infty}\M)$ and therefore finitely generated.
\end{proof}
\begin{theorem}
\label{divides11}
$\textup{Char}(A_{\infty,\chi})\mid \textup{Char}((\overline{E}/\overline{C})_{\chi})$.
\end{theorem}
\begin{proof}The main argument of this proof is analogous to \cite[page 371]{washington}.
From \eqref{iiia} and \eqref{iiib} we obtain that $\prod_{i=1}^kg_iv_{\mathfrak{Q}_{k+1}}(\kappa(q_1\dots\q_{k+1}))_{\chi}=\eta h'_{\chi}   \mod 2^n$, where $\eta=\tilde{u}\vert H\vert ^{k+1}2^{k(2d+c_3)+d+c_2}(\gamma'_{n_0}-1)^{c_1+\sum_{j=1}^{k}c_1^j}$ for some unit $\tilde{u}$.  It follows that $\prod_{i=1}^kg_i$ divides $ \eta h'_{\chi}$ in $\Lambda_{\chi,m+n-m'}/2^n\Lambda_{\chi,m+n-m'}$. For every $n$ we can find an element $z_n$ such that $\prod_{i=1}^kg_i z_n=\eta h'_{\chi}$ in $\Lambda_{\chi,n}/2^n\Lambda_{\chi,m+n-m'}$. The $z_n$'s have a converegent subsequence and we obtain that $\prod_{i=1}^kg_i\mid \eta h'_{\chi}$ in $\Lambda_{\chi}$. By Lemma \ref{leo} and Corollary \ref{mu_2} $\textup{Char}(A_{\infty,\chi})$ is coprime to $\eta$ and the claim follows.
\end{proof}
\begin{remark}
Vigui\'{e} proves in \cite{vg2} as well that there is a power of $p$ such that $\textup{Char}(A_{\infty,\chi})\mid p^v\textup{Char}((\overline{E}/\overline{C})_{\chi})$. Using Corollary \ref{mu_2} this implies Theorem \ref{divides11}. He follows the proofs of Bley very closely as well but he uses a slighly different version of Theorem \ref{tchebotarev} and does not use the relation between $\Gamma$ and $\Gamma'$ as we did here in Section \ref{sec;components}. 
\end{remark}
\begin{corollary}
\label{dividesganz}
$\textup{Char}(A_{\infty})\mid \textup{Char}(\overline{E}/\overline{C})$
\end{corollary}
\begin{proof}
As Theorem $\ref{divides11}$ holds for all characters and $\textup{Char}(A_{\infty})$ is coprime to $2$ this is immediate.
\end{proof}

\section{Characteristic ideals and the main conjecture}
Consider the exact sequence
\[0\to\overline{E}/\overline{C}\to U_{\infty}/\overline{C}\to X\to A_{\infty}\to 0,\] where $X=\Gal(\Omega/\L_{\infty})$. Then 
\begin{equation} \label{chareq}\textup{Char}(A_{\infty}) \textup{Char}(U_{\infty}/\overline{C})=\textup{Char}(X) \textup{Char}(\overline{E}/\overline{C}).
\end{equation} From Corollary \ref{dividesganz} we deduce \begin{equation}
\label{teiler}
\textup{Char}(X)\mid \textup{Char}(U_{\infty}/\overline{C}).\end{equation} In the following we will establish a relation between $p$-adic $L$-functions and elliptic units to show that $\textup{Char}(X)$ is in fact equal to $\textup{Char}(U_{\infty}/\overline{C})$.

Let $u\in U_{\infty}$ and let $g_u(w)$ be the Coleman power series of $u$ (see \cite[I Theorem 2.2]{Shalit}). Let $\tilde{g}_u(W)=\log g_u(W)-\frac{1}{p}\sum_{w\in \widehat{E}_{\p}}\log g_u(W\oplus w)$. There exists a measure $\nu_u$ on $\Z_p^{\times}$ having $\tilde{g}_u\circ \beta^v$ as characteristic series \cite[I 3.4]{Shalit}. Recall that $\mathcal{D}_{\p}=\I_{\p}(\zeta_m)$ and let $\Lambda(\mathcal{D}_{\p},\Gamma'\times H)$ be the algebra of $\mathcal{D}_{\p}$-valued measures on $\Gamma'\times H$.
Define
\[\iota(\f) \colon U_{\infty}\widehat{\otimes}\mathcal{D}_{\p}\to \Lambda(\mathcal{D}_{\p},\Gamma'\times H),\quad u\mapsto \sum_{\sigma\in \Gal(\F/\K)}\nu_{\sigma u}\circ \sigma.\]
Note that this construction of measures coincides with the one from section 2 for elliptic units. 
\begin{lemma}
\label{iotapseudo}
$\iota(\f)$ is a pseudoisomorphism.
\end{lemma}
\begin{proof}
By \cite[I Theorem 3.7]{Shalit} it suffices to prove that the completion of $\L_{\infty}$ at all primes above $\p$ contains only finitely many $2$-power roots of unity. But this follows from \cite[proof of Proposition 4.3.10]{Kezuka-thesis}. 
\end{proof}

For every $\g\mid \f $ there is a map $\iota(\g)\colon U_{\infty}(\K(\g\p^{\infty}))\to\Lambda(\mathcal{D}_{\p},\Gal(\K(\g\p^{\infty}/\K)))$. 
Note that there are natural restriction and corestriction maps $\pi_{\f,\g}$ and $\eta_{\f,\g}$ such that $\pi_{\f,\g}\circ\iota(\f)=\iota(\g)\circ N_{\f,\g}$ and $\iota(f)\circ \textup{inclusion}=\eta_{\f,\g}\circ \iota(\g)$, where inclusion is the natural map $U_{\infty}(\K(\g\p^{\infty}))\to U_{\infty}$ (see \cite[page 100]{Shalit} for details).
\begin{proposition}
Let $\chi$ be a character of $H$ and $\g$ the prime to $\p$-part of the conductor. The module $\chi\circ\iota(\f) (\overline{C(\f))}$ is pseudoisomorphic to $\chi(\nu(\g) \Lambda(\mathcal{D}_{\p},\Gamma'\times H))$, if $\chi$ is non-trivial. If $\chi$ is trivial $\chi\circ \iota(\f) (\overline{C(\f)})$ is pseudoisomorphic to $(\gamma'-1)\chi(\nu(1)\Lambda(\mathcal{D}_{\p},\Gamma'\times H))$.
\end{proposition}
\begin{proof}
Analogous to \cite[III Lemma 1.10]{Shalit}.
As $\chi$ has conductor dividing $\g\p^2$ it follows that $\chi\circ \iota(\f) (\overline{C(\f)})=\chi\circ \pi_{\f,\g}\circ\iota(\f)(\overline{C(\f)})=
\chi\circ\iota (\g)N_{\f,\g}\overline{C(\f)}$.
Assume first that $\g\neq 1$.
 It is immediate that $\sum_{\sigma\in \Gal(\K(\g\p^{\infty}))/\K(\mathfrak{h}\p^{\infty}))}\chi(\sigma)=0$. Hence, \begin{equation}
\label{chi}
\chi\circ \iota(g)(\overline{C_{\g}})=\chi\circ\iota(\g)(\overline{C(\g)}).
\end{equation}  If $\omega_{\g}=1$ we can construct the measure $\nu(\g)$ as in section 2 and obtain that $\iota(\g)(\overline{C_{\g}})$ is the ideal generated by $\mathcal{J}\nu(\g)$ where $\mathcal{J}$ is the ideal generated by all the $\mu_{\alpha}=\nu/\nu_{\alpha}$. If $\omega_{\g}\neq 1$ there exists an integer $k$ such that $\omega_{\g^k}=1$ and then we can define the measure $\nu(\g^k)$. But by \eqref{smallermeas} we have $\nu(\g)$ is the restriction of $\nu(\g^k)$ and $N_{\g^k,\g}$ is surjective on the elliptic units. So in both cases the image under $\iota(\g)$ is precisely $\mathcal{J}\nu(\g)$.

 If the norm $N_{\f,\g}:\overline{C(\f)}\to \overline{C(\g)}$ is not surjective   it follows that the cokernel of the module $\chi\circ \iota(g)\circ N_{\f,\g}(\overline{C(\f)})$ in $\chi\circ \iota(\g)(\overline{C(\g)})$ is annihilated by $[\K(\f\p^{\infty}):\K(\g\p^{\infty})]$ and the product $\prod_{l\mid \f,l\nmid \g} (1-\chi(\sigma_{l})\sigma_l^{-1}\mid_{\Gamma'})$.  These elements are certainly coprime and we see that $\chi\circ\iota(\f)(\overline{C(\f)})\sim \chi\circ \iota(\g)(\overline{C_{\g}})$ due to \eqref{chi}, where $A\sim B$ means that $A$ and $B$ are pseudoisomorphic. But the $\chi(\mu_{\alpha})$ are coprime due to Theorem \ref{shalit412} and the claim follows for $\g\neq (1)$.

\smallskip Assume now that $\g=(1)$. Let $\tau\in \Gal(\K(\p^{\infty})/\K)$ then the elements $\xi_{\alpha,\sigma}(P^{\sigma}_n)^{\tau-1}$ are norms of elliptic units in $\K(\mathfrak{h}\p^{n})$, where $\mathfrak{h}$ is a prime having Artin symbol $\tau^{-1}$ in $\Gal(\K(\p^{\infty})/\K)$. It follows that the element $\xi_{\alpha,\sigma}(P^{\sigma}_n)^{\tau-1}$ corresponds to the measure $\nu_{\alpha}(\tau-1)\nu(1)$ under $\iota(1)$. The group $\overline{C_{(1)}}$ is generated by products $\prod_{i=1}^s\xi_{\alpha_i,\sigma}(P^{\sigma}_n)^{m_i}$ with $\sum m_i(N\alpha_i-1)=0$. Let $\nu_{\pi}$ be the measure corresponding to such a product. Then we obtain $(\tau-1)\nu_{\pi}=\sum m_i\mu_{\alpha_i}(\tau-1)\nu(1)$. As $(\tau-1)\nu(1)$ is not contained in the augmentation of $\Lambda(\mathcal{D}_{\p},\Gal(\K(\p^{\infty})/\K))$ we obtain that the ideal generated by the $\sum m_i\mu_{\alpha_i}$ is the augmentation ideal and that $\iota((1))(\overline{C_{(1)}})=\mathcal{A}\nu((1))$, where $\mathcal{A}$ denotes the augmentation of $\Lambda(\mathcal{D}_{\p},\Gal(\K(\p^{\infty})/\K))$. Analogously to the case $\g\neq (1)$ we can conclude that $\chi\circ\iota((1))\circ N_{\f,(1)} (\overline{C(\f)})$ has finite index in $\chi\circ\iota((1)) (\overline{C_{(1)}})$. Hence, it suffices to consider the image $\chi\circ\iota((1))(\overline{C_{(1)}})$. 
If $\chi$ is a non-trivial character, then $\chi(\mathcal{A})$ contains $\chi(\tau)-1$ as well as $\gamma'-1$. Thus $\chi\circ \iota((1))(\overline{C_{(1)}})\sim \chi(\nu(1))$. If $\chi$ is the trivial character then $\chi\circ \iota((1))(\overline{C_{(1)}})=(\gamma'-1)\chi(\nu(1))$. 
\end{proof}

\begin{corollary}
\label{cor:lfunction}
Let $F(w,\chi)$ be the Iwasawa function associated to $L_{\p}(s,\chi)$ defined in Definition \ref{padicdef}. Then
$\textup{Char}((U_{\infty}/\overline{C})_{\chi})=F(w,\chi)$.
\end{corollary}
\begin{proof}
Let $\g$ be such that the conductor of $\chi$ is $\g\p$ or $\g$.
By Lemma \ref{iotapseudo} we see that $\textup{Char}(U_{\infty}/\overline{C})$ equals $\textup{Char}(\Lambda(\mathcal{D}_{\p},\Gal(\K(\p^{\infty})/\K))/\iota(\f)(C(\f)))$. But the latter equals $\chi(\nu(\g))$ if $\chi$ is non-trivial and $(1-\gamma)\chi(\nu(1))$ if $\chi$ is trivial. But these are precisely the measures used to define $L_{\p}(s,\chi^{-1})$. As $\int_{\mathcal{G}} \kappa^s\chi d (1-\gamma)^e\nu(\g)=\int_{\Gamma'}\kappa^sd(1-\gamma)^e \chi(\nu(\g))$, where $e=1$ if $\chi$ is trivial and $e=0$ in all other cases, the claim follows.
\end{proof}

\subsection{Matching the invariants}
In the following we will show how the $\lambda$-and $\mu$-invariants of $F(w,\chi)$ match with the ones of $X$. This section follows closely Section 4 of \cite{mu}. 
Recall that $\L_n=\K(\f\p^{n+2})$.
To start with we need the following result from  \cite{mu}. Let $t$ be such that $\K_t=\F\cap \K_{\infty}$.

\begin{corollary}\label{shalit28} If $G \in \Z_p[[\Gamma']]$ is a characteristic power series for $\Gal(\M(\L_{\infty})/\L_{\infty})$, then for all sufficiently large $n$ one has
\[ \mu(G) 2^{t+n-1}+\lambda(G)=1+\ord_2 \left[ \frac{h\left(\L_{n}\right) R_{\p}\left(\L_{n}\right)}{\omega\left(\L_{n}\right) \sqrt{\Delta_{\p}\left(\L_{n}/\K\right)}}\mathlarger{\mathlarger{/}}\frac{h\left(\L_{n-1}\right) R_{\p}\left(\L_{n-1}\right)}{\omega\left(\L_{n-1}\right) \sqrt{\Delta_{\p}\left(\L_{n-1}/\K\right)}}\right].\]
\end{corollary}

 Note that $\mathcal{D}_{\p}[[\Gamma']] \cong \mathcal{D}_{\p}[[w]]$. Consider any character $\rho$ of $\Gamma'$ of finite order. We say level $(\rho)=m$ if $\rho\left(\left(\Gamma'\right)^{2^m}\right)=1$, but $\rho\left(\left(\Gamma'\right)^{2^{m-1}}\right) \neq 1$.
 
 To determine the invariants of the Iwasawa function $F(w,\chi)$ we need the following two results (\cite[Chapter III, Lemma 2.9]{Shalit} and \cite[Chapter III, Proposition 2.10]{Shalit}). 

\begin{lemma}\label{shalit29} For any power series $F \in \mathcal{D}_{\p}[[w]]$ and all sufficiently large $n$, one has
\[ \mu\left(F\right) 2^{n+t-1}+\lambda(F)=\mathrm{ord}_2 \left\lbrace \prod\limits_{\mathrm{level}(\rho)=t+n} \rho(F) \right\rbrace,\]
where $\rho(F)$ means that the action of $\rho$ is extended to $\mathcal{D}_{\p}[[\Gamma']]$ by linearity and $\mathrm{ord}_p$ is the valuation on $\C_p$ normalized by taking $\ord_2(2)=1$.
\end{lemma}

\begin{proposition}\label{shalit210} For any ramified character $\eps$ of $\Gal(\F_{\infty}/\K)$, we let $\g$ be the conductor of $\eps$ and $g$ the least positive integer in $\g \cap \Z$. We define $G(\eps)$ as in Theorem \ref{shalit412} and we define $S_p(\eps)$ by
\[ S_2(\eps)=-\frac{1}{12g\omega_{\g}} \sum\limits_{\sigma \in \Gal(\K(\g)/\K)} \eps^{-1}(\sigma)\log \varphi_{\g}(\sigma).\] 
Let $A_n$ be the collection of all $\eps$ for which $n$ is the exact power of $\p$ dividing their Artin conductor. Then for all sufficiently large $n$ one has
\[
\ord_2\left(\prod\limits_{\eps \in A_{n+2}} G(\eps) S_p(\eps)\right)=\ord_2 \left[ \frac{h\left(\L_n\right) R_{\p}\left(\L_n\right)}{\omega\left(\L_n\right) \sqrt{\Delta_{\p}\left(\L_n/\K\right)}}\mathlarger{\mathlarger{/}}\frac{h\left(\L_{n-1}\right) R_{\p}\left(\L_{n-1}\right)}{\omega\left(\L_{n-1}\right) \sqrt{\Delta_{\p}\left(\L_{n-1}/\K\right)}}\right].\]
\end{proposition}

Using \cite[Theorem 5]{mu}, for a character $\rho$ of $\Gamma'$ of sufficiently large finite order, one has
\[ \rho(F(w,\chi^{-1})) \sim \begin{cases} G(\chi\rho)S_2(\chi\rho) &\mbox{if } \chi \neq 1;\\
\left(\rho(\gamma_0)-1\right)G(\chi\rho)S_2(\chi\rho) &\mbox{if } \chi=1,\end{cases}\]
where $u \sim v$ denotes the fact that $u/v$ is a $\p$-adic unit.
Let 
\[F=\prod\limits_{\chi \in \widehat{H}} F(w,\chi).\]
It follows that for all sufficiently large $n$ one has
\begin{eqnarray}\label{dsdsasdasa}
\prod\limits_{\mathrm{level}(\rho)=t+n} \rho(F) \sim 2\prod\limits_{\stackrel{\eps=\chi \rho}{\mathrm{level}(\rho)=t+n}} G(\eps)S_p(\eps),
\end{eqnarray}
since in the product on the right hand side we range over all $\chi$ (including $\chi=1$) and 
\[\prod\limits_{\mathrm{level}(\rho)=t+n} \left(\rho(\gamma_0)-1\right)=2.\]

As $\L_0/\F$ is ramified at $\p$ of degree $2$ and $\p$ is unramified in $\F/\K$ we see that a character in $A_{n+2}$ is of level $t+n$.
Combining Corollary \ref{shalit28}, Lemma \ref{shalit29} and \eqref{dsdsasdasa} we obtain that
\[\mu\left(F\right) 2^{n+t-1}+\lambda(F)=\mu(G) 2^{t+n-1}+\lambda(G)\quad \forall n\gg0\]
This implies together with Theorem \ref{mu1}
\begin{theorem}
\label{sameinvariants}
$\mu(G)=\mu(F)=0$ and $\lambda(G)=\lambda(F)$.
\end{theorem}

\subsection{Proving the main conjecture}
In this section we use all the results proved before to prove the main conjecture.

\begin{lemma}
\label{main2}
$\textup{Char}(X)=\textup{Char}(U_{\infty}/\overline{C})$ and $\textup{Char}(A_{\infty,\chi})=(\overline{E}/\overline{C})_{\chi}$.
\end{lemma}
\begin{proof}
The first claim follows directly from \eqref{teiler} and Theorem \ref{sameinvariants}. From \eqref{chareq} we also obtain that $\textup{Char}(A_{\infty})=\textup{Char}(\overline{E}/\overline{C})$. Further Theorem \ref{divides11} establishes that $\textup{Char}(A_{\infty,\chi})$ divides $\textup{Char}(\overline{E}/\overline{C})_{\chi}$. Both togehther imply the second claim.
\end{proof}
This has also the following consequence:
\begin{theorem}
\label{main}
$\textup{Char}(X_{\chi})=\textup{Char}((U_{\infty}/\overline{C})_{\chi})$ for any $\chi$.
\end{theorem}
\begin{proof}
For any $\Lambda$-module we denote by $M^{\chi}$ the largest submodule in $M\otimes_{\Z_p}\Z_p(\chi)$ on which $H$ acts via $\chi$. By \cite[page 5]{Tsuji} there exists a homomorphism between $M_{\chi}$ and $M^{\chi}$ such that the kernel and the cokernel are annihilated by $\vert H\vert$.
As none of the characteristic ideals involved is divisible by $2$ we can consider the charactersitc ideals of $M^{\chi}$ instead of $M_{\chi}$ for any  $M$ in $\{A_{\infty}, U_{\infty}/\overline{C}, X, \overline{E}/\overline{C}\}$. 
The sequence  \[0\to (\overline{E}/\overline{C})^{\chi}\to (U_{\infty}/\overline{C})^{\chi}\to X^{\chi}\]
is exact. Let $e_{\chi}$ in $\Q_p(\chi)[H]$ be the idempotent induced by the character $\chi$. Then $e_{\chi}\vert H\vert $ is an element in $\Z_p(\chi)[H]$. In particular, $e_{\chi}\vert H\vert M\subset M^{\chi}$. It follows that the cokernel of the natural homomorphism $\phi_{\chi}:X^{\chi}\to A_{\infty}^{\chi}$ is annihilated by $\vert H\vert$. As $A_{\infty}$ has bounded rank it follows that the $\coker(\phi_{\chi})$ is finite. The module $\ker(\phi_{\chi})$ equals $X^{\chi}\cap im(U_{\infty}/\overline{C})$. Again the exponent of $X^{\chi}\cap im((U_{\infty}/\overline{C}))/im((U_{\infty}/ \overline{C})^{\chi})$ is bounded by $\vert H \vert$. Hence, 
$\textup{Char}(A_{\infty}^{\chi})\textup{Char}(im((U_{\infty}/\overline{C})^{\chi})=\textup{Char}(X^{\chi})$. Using the exactness of the sequence above we obtain \[\textup{Char}(A_{\infty}^{\chi})\textup{Char}((U_{\infty}/\overline{C})^{\chi})=\textup{Char}((\overline{E}/\overline{C})^{\chi})\textup{Char}(X^{\chi}).\] The claim follows now from Lemma \ref{main2}.
\end{proof}

The second claim of Lemma \ref{main2} and  Theorem \ref{main} prove Theorem \ref{mainmain} for $\L_{\infty}$.

\smallskip
\begin{center}

\textbf{Acknowledgements}\end{center}
The author would like to thank S\"oren Kleine for his comments on preliminary versions of this paper.

\end{document}